\documentclass[reqno, 11pt, a4paper]{amsart}
\usepackage{amsmath, amssymb}  
\usepackage{amsthm}
\usepackage{graphicx, color}
\usepackage[foot]{amsaddr}
\usepackage{mathtools}
\usepackage{mathrsfs, enumerate}
\usepackage{hyperref}
\usepackage[truedimen,top=3truecm, bottom=3truecm, left=2.5truecm, right=2.5truecm, includefoot]{geometry} 
\begin{document}

\title[Equilibrium fluctuations for totally asymmetric interacting particles]{Equilibrium fluctuations for totally asymmetric interacting particle systems}
\author[Kohei Hayashi]{Kohei Hayashi}
\address{Graduate School of Mathematical Sciences, The University of Tokyo, Komaba, Tokyo 153-8914, Japan.}
\email{kohei@ms.u-tokyo.ac.jp}
\maketitle

\newtheorem{definition}{Definition}[section]
\newtheorem{theorem}{Theorem}[section] 
\newtheorem{lemma}[theorem]{Lemma}
\newtheorem{corollary}[theorem]{Corollary}
\newtheorem{proposition}[theorem]{Proposition}
\newtheorem{remark}{Remark}[section]
\newtheorem{assumption}{Assumption}[section]
\newtheorem{example}{Example}[section]

\makeatletter
\renewcommand{\theequation}{%
\thesection.\arabic{equation}}
\@addtoreset{equation}{section}
\makeatother

\makeatletter
\renewcommand{\p@enumi}{A}
\makeatother

%%%%%%%%%%%%%%%%%%%%%%%%%%%%%%%%%%%%%%%%%%
%   Define Roman numbers here (I, II, III, IV, etc)    %
%%%%%%%%%%%%%%%%%%%%%%%%%%%%%%%%%%%%%%%%%%
\newcounter{num}
\newcommand{\Rnum}[1]{\setcounter{num}{#1}\Roman{num}}

%%%%%%%%%%%%%%%%%%%%%%%%%%%%%%%%%
%    abstract    
%%%%%%%%%%%%%%%%%%%%%%%%%%%%%%%%%
\begin{abstract}
We study equilibrium fluctuations for a class of totally asymmetric zero-range type interacting particle systems. As a main result, we show that density fluctuation of our process converges to the stationary energy solution of the stochastic Burgers equation. As a special case, microscopic system we consider here is related to $q$-totally asymmetric simple exclusion processes ($q$-TASEPs) and our scaling limit corresponds to letting the quantum parameter $q$ to be one. 
\end{abstract}

%%%%%%%%%%%%%%%%%%%%%%%%%%%%%%%%%%%%%%
%     Introduction                   % 
%%%%%%%%%%%%%%%%%%%%%%%%%%%%%%%%%%%%%%
\section{Introduction}
In this paper, we have an interest in Kardar-Parisi-Zhang (KPZ) equation, which is a stochastic partial differential equation of unknown function $h = h (t, x ) $ where $ ( t , x ) \in [ 0, \infty ) \times \mathbb{R} $ with the form 
\begin{equation}
%\tag{KPZ}
\label{eq:KPZintro}
\partial_t h = \nu \partial_x^2 h + \lambda (\partial_x h )^2 + \sqrt{D} \dot{W}  ( t, x )  . 
%,\text{ in } \mathbb{R}_+ \times \mathbb{R} 
\end{equation}
Here $\nu, D > 0 $ and $\lambda \in \mathbb{R}$ are constants and $\dot{W} (t, x  ) $ is the space-time white noise. Or equivalently, we focus on its tilt $u = \partial_x h$ which satisfies the stochastic Burgers equation (SBE) 
\begin{equation}
\label{eq:SBEintro}
\partial_t u = \nu \partial_x^2 u + \lambda \partial_x u^2 + \sqrt{D} \partial_x \dot{W}  ( t, x ) .
\end{equation}
Throughout this paper, we only consider the one-dimensional setting. KPZ equation is introduced in \cite{kardar1986dynamic} as a model to describe random interface evolution. The main interest of this paper is universality of interface growth. Before discuss in detail the universality, we mention solution theory of KPZ equation. Looking the SBE \eqref{eq:SBEintro}, the solution $u$ is typically expected to have the same regularity as the space-time white noise. In particular, it takes values on distribution and thus the non-linear term $\partial_x u^2$ cannot be defined naively. As a consequence, the equation \eqref{eq:SBEintro} and also \eqref{eq:KPZintro} are called singular type equation in this sense. One way to use Cole-Hopf transformation $Z = \exp ((\lambda /\nu)h )$. Then the transformed process $Z$ satisfies the stochastic heat equation with multiplicative noise
\begin{equation}
\label{eq:SHEintro}
\partial_t Z = \nu \partial_x^2 Z + \frac{\lambda \sqrt{D}}{\nu} Z\dot{W} (t, x ) .
\end{equation}
Now we can give a meaning to the solution of \eqref{eq:SHEintro} in a classical way and then the solution to KPZ equation can be defined by $h = (\nu / \lambda) \log Z$, which is called the Cole-Hopf solution. However, such a good transformation is restrictive and a solution theory which directly give a meaning to singular stochastic differential equations is preferable. To prove well-posedness of \eqref{eq:KPZintro} itself without using the Cole-Hopf transformation, a renormalization procedure which roughly subtract ``$-\infty$'' from the singular term is needed. Such a renormalization is conducted in a mathematically rigorous way in \cite{hairer2013solving} for the first time, and well-posedness of KPZ equation is proved there. And then the solution theory is generalized as the regularity structure theory in \cite{hairer2014theory} covering more wide range of singular stochastic partial differential equations. On the other hand, the paper \cite{gubinelli2015paracontrolled} introduced the notion of paracontrolled calculus and then global well-posedness of KPZ equation is shown in \cite{gubinelli2017kpz} based on paracontrolled calculus. Both solution theories are established as generalization of rough path theory, and particularly the solutions are constructed based on a pathwise approach rather than a probabilistic one. For a probabilistic construction, though restricted on the stationary case, the notion of energy solution is introduced in \cite{gonccalves2014nonlinear} as a martingale problem formulation and existence of the solution is shown. Then uniqueness of energy solution is proved in \cite{gubinelli2018energy}.

The main interest of this paper is to derive \eqref{eq:KPZintro} as an equation which describes macroscopic interface evolution, by taking scaling limits of microscopic models. Until now, several microscopic models from which KPZ equation is derived by scaling limits are known. In particular, for microscopic models under equilibrium state, the notion of energy solution gives us a robust way to derive KPZ equation as scaling limits. Here we briefly review results on the universality of KPZ equation for stationary models. (See \cite{corwin2012kardar} for progress in this decades containing also non-stationary cases.) As to stationary case, \cite{bertini1997stochastic} is a celebrating result, which proved that density fluctuation of simple exclusion processes with weak asymmetric jump rates converges to the Cole-Hopf solution of SBE. After that, \cite{gonccalves2014nonlinear} generalized the result of \cite{bertini1997stochastic} to wider class of jump rates and remarkably they established a robust way to derive KPZ equation without using Cole-Hopf transformation: \cite{gonccalves2015stochastic} for interacting particle systems containing zero-range processes, \cite{diehl2017kardar} for a system of stochastic differential equations and \cite{jara2019scaling} for the Sasamoto-Spohn model, which is originally introduced in \cite{sasamoto2009superdiffusivity}. Other important class from which KPZ equation is derived is directed polymers, which is introduced in \cite{huse1985pinning} and mathematically analyzed in \cite{imbrie1988diffusion} for the first time. As to the stationary case, recently \cite{jara2020stationary} derived the stochastic Burgers equation from free-energy fluctuation of the stationary O'Connell-Yor model (\cite{o2001brownian}). On the other hand, a some relation between the O'Connell-Yor polymer and an interacting particle system is pointed out: the $q$-deformation of totally asymmetric simple exclusion process ($q$-TASEP, in short) with parameter $q \in (0, 1)$ is introduced in \cite{borodin2014macdonald} and moreover it is proved that the $q$-TASEP in some sense converges to the O'Connell-Yor polymer as $q \to 1$. (See also \cite{borodin2014duality}.) From this degeneration result, it is expected that the stochastic Burgers equation can also be derived by scaling limits of $q$-TASEPs. In this paper we consider a class of totally asymmetric interacting particle systems where particles on one dimensional lattice move only to one direction, containing $q$-TASEP model as a special case. As a main result, we show that the stochastic Burgers equation is derived from our model. 

%Finally we explain how this paper is organized. In Section \ref{sec:model}, after recalling the notion of stationary energy solution of KPZ equation, we introduce our interacting particle model and then state the main result. Our main theorem is that density fluctuation of the microscopic particle system converges to the stationary energy solution of the stochastic Burgers equation. Before give a proof of the main theorem, we give an outline in Section \ref{sec:outline}. Section \ref{sec:BG2}... In Section \ref{sec:tightness}, we show tightness of the family of microscopic fluctuation fields and then in Section \ref{sec:limitpt} we characterize limit points. Section \ref{sec:expansion}...

%%%%%%%%%%%%%%%%%%%%%%
%  Model and result  %
%%%%%%%%%%%%%%%%%%%%%%
\section{Main results}
\label{sec:model}

\subsection{Stationary energy solution of KPZ/SBE}
%\label{sec:energysol}
In the sequel, we write $\mathbb{R}_+ \coloneqq [ 0, \infty ) $. Let $\nu, D > 0  $ and $\lambda \in \mathbb{R}$ be fixed constants and consider $(1 + 1 )$-dimensional KPZ equation  
\begin{equation}
\label{KPZ}
\partial_t h = \nu \partial_x^2 h  + \lambda  (\partial_x h)^2 +  \sqrt{D} \dot{W} (t, x )  
\quad \text{ in } \mathbb{R}_+ \times \mathbb{R} .
\end{equation}
Then recall that the tilt $u = \partial_x h $ satisfies the stochastic Burgers equation
\begin{equation}
\label{SBE}
\partial_t u = \nu \partial_x^2 u  + \lambda \partial_x u^2 + \sqrt{D} \partial_x \dot{W} (t, x )  
\quad \text{ in } \mathbb{R}_+ \times \mathbb{R} .  
\end{equation}
As a preliminary we recall the notion of \textit{stationary energy solution}. The same formulation can be applied for KPZ equation \eqref{KPZ} so that we focus only on the stochastic Burgers equation \eqref{SBE}. Now we begin with the definition of stationarity. 

%\begin{definition}[White noise]
%Let $(\Omega , \mathcal{F} , (\mathcal{F}_t )_{t \ge 0  } , \mathbb{P} ) $ be a filtered probability space. For an adapted process $W$ with trajectories in $C (\mathbb{R}_+ , \mathcal{S}^\prime (\mathbb{R} ) ) $, we say that $\dot{ W }_t = \partial_t W $ is a \textit{space-time white noise} if for all $\varphi \in \mathcal{S} (\mathbb{R} ) $ the process $\{ W_t (\varphi ) \} $ is a Brownian motion in the filtration $(\mathcal{F}_t )_{ t \ge 0 } $ with variance $\mathbb{E} [ W_t (\varphi ) ] = t \| \varphi \|^2_{ L^2 (\mathbb{R} ) } $ for all $ t \ge 0 $. A \textit{(space) white noise with variance $\sigma^2 $} is a random variable $\eta $ with values in $\mathcal{S}^\prime (\mathbb{R} ) $ such that $\{ \eta (\varphi ) : \varphi \in \mathcal{S} (\mathbb{R} ) \} $ is a centered Gaussian process with covariance $\mathbb{E} [ \eta (\varphi ) \eta (\psi ) ] = \sigma^2 \langle \varphi , \psi \rangle_{ L^2 (\mathbb{R} ) } $ for every $\varphi, \psi \in \mathcal{S} (\mathbb{R} ) $. 
%\end{definition}

\begin{definition}
We say that an $\mathcal{S}^\prime (\mathbb{R})$-valued process $u = \{ u_t  : t \in [0,T] \} $ satisfies condition \textbf{(S)} if for all $t \in [0,T]$, the random variable $u_t$ has the same distribution as space white noise with variance $D/(2\nu)$. 
\end{definition}

For a process $u = \{ u_t: t \in [0,T]\}$ satisfying the condition \textbf{(S)}, we define 
\[
\mathcal{A}^\varepsilon_{ s, t } (\varphi ) = \int_s^t \int_{\mathbb{R} } u_r (\iota_\varepsilon (x; \cdot) )^2 \partial_x \varphi (x ) dx dr .  
\]
for every $0 \le s < t \le T $, $\varphi \in \mathcal{S} (\mathbb{R} ) $ and $\varepsilon > 0 $. Here we defined the function $\iota_\varepsilon (x ; \cdot ) : \mathbb{R} \to \mathbb{R}  $ by $\iota_{ \varepsilon } (x ; y) =  \varepsilon^{ - 1 } \mathbf{1}_{ [ x , x  + \varepsilon ) } (y) $ for each $x \in \mathbb{R} $. 

\begin{definition}
Let $u = \{ u_t :t \in [0,T]\}$ be a process satisfying the condition \textbf{(S)}. We say that the process $u$ satisfies the energy estimate if there exists a constant $\kappa > 0$ such that:
\begin{itemize}
\item[\textbf{(EC1)}]
For any $\varphi \in \mathcal{S} (\mathbb{R} )$ and any $0 \le s < t \le T$,
\begin{equation*}
\mathbb{E}_n \bigg[ \bigg| \int_s^t u_r (\partial_x^2 \varphi ) dr  \bigg|^2 \bigg] \le \kappa  (t- s ) \| \partial_x \varphi \|^2_{ L^2(\mathbb{R} ) } .
\end{equation*}

\item[\textbf{(EC2)}]
For any $\varphi \in \mathcal{S} (\mathbb{R} )$, any $0 \le s < t \le T$ and any $0 < \delta < \varepsilon < 1 $,  
\begin{equation*}
\mathbb{E}_n \big[ \big| \mathcal{A}^\varepsilon_{ s, t } (\varphi ) - \mathcal{A}^\delta_{ s, t } (\varphi ) \big|^2 \big] 
\le \kappa \varepsilon (t- s ) \| \partial_x \varphi \|^2_{ L^2(\mathbb{R} ) } .
\end{equation*}
\end{itemize}
\end{definition}

Then the following result is proved in \cite{gonccalves2014nonlinear}. 

\begin{proposition}
\label{nonlinear}
Assume $\{ u_t : t \in [0, T ] \} $ satisfies the conditions \textbf{(S)} and \textbf{(EC2)}. Then there exists an $\mathcal{S}^\prime (\mathbb{R} )$-valued process $\{ \mathcal{A}_t : t \in [0, T ] \} $ with continuous trajectories such that  
\[
\mathcal{A}_t (\varphi ) = \lim_{ \varepsilon \to 0 } \mathcal{A}^\varepsilon_{ 0, t } (\varphi) 
\]
in $L^2 $ for every $t \in [0, T ] $ and $\varphi \in \mathcal{S} (\mathbb{R} ) $.  
\end{proposition}

By this proposition, thinking the singular term $\partial_x u^2 $ is given by this quantity, we can define a solution of \eqref{SBE} as follows.

\begin{definition}
\label{def:energysol}
We say that an $\mathcal{S}^\prime(\mathbb{R})$-valued process $u=\{u (t, \cdot) : t\in [0,T] \}$ is a stationary energy solution of the stochastic Burgers equation \eqref{SBE} if 
\begin{enumerate}
\item The process $u$ satisfies the conditions \textbf{(S)}, \textbf{(EC1)} and \textbf{(EC2)}. 
\item For all $\varphi \in \mathcal{S} (\mathbb{R} )$, the process 
\[
u_t ( \varphi ) - u_0 (\varphi ) - \nu \int_0^t u_s (\partial_x^2 \varphi ) ds - \mathcal{A}_t (\varphi ) 
\]
is a martingale with quadratic variation $D \| \partial_x^2 \varphi \|^2_{ L^2 (\mathbb{R} ) } t $ where $\mathcal{A}$ is the process obtained in Proposition \ref{nonlinear}. 
\item For all $\varphi \in \mathcal{S} (\mathbb{R} )$, writing $\hat{ u }_t = u_{ T - t } $ and $\hat{ \mathcal{A} }_t = - (\mathcal{A}_T - \mathcal{A}_{ T- t })$, the process
\[
\hat{ u }_t ( \varphi ) - \hat{ u }_0 (\varphi ) - \nu \int_0^t \hat{ u }_s (\partial_x^2 \varphi ) ds - \hat{ \mathcal{A} }_t (\varphi ) 
\]
is a martingale with quadratic variation $D \| \partial_x^2 \varphi \|^2_{ L^2 (\mathbb{R} ) } t $. 
\end{enumerate}
\end{definition}

Then it is proved that there exists a unique-in-law stationary energy solution of \eqref{SBE}. Existence was shown in \cite{gonccalves2014nonlinear} and then uniqueness was proved in \cite{gubinelli2018energy}.

\subsection{Model and result}
Throughout this paper we write $\mathbb{N} = \{ 1, 2, \ldots \}$ and $\mathbb{Z}_+ = \{ 0, 1, \ldots \}$. Let $\mathscr{X} = \mathbb{Z}_+^{ \mathbb{Z} } $ be a configuration space and we consider Markov processes which takes values on $\mathscr{X}$. We write an element in the configuration space $\mathscr{X} $ by Greek letters $\eta = \{ \eta_j : j \in \mathbb{Z} \}$ where $\eta_j  $ denotes the number of particles on a site $j \in \mathbb{Z}$. Let $c:\mathbb{Z}_+ \to \mathbb{R}_+$ be such that $c(0)=0$ and take $p_n, q_n \in [0,1]$ satisfying $p_n + q_n =1$. Then zero-range process is a Markov process with generator 
\[
 f (\eta)
\mapsto n^2 \sum_{j \in \mathbb{Z}} p_n c(\eta_j) \nabla_{j,j-1} f(\eta)
+ n^2 \sum_{j \in \mathbb{Z}} q_n c(\eta_j) \nabla_{j,j+1} f(\eta)
\]
acting on each local function $f : \mathscr{X} \to \mathbb{R}$. Here $\nabla_{ j , j + 1 } f ( \eta ) = f ( \eta^{ j, j + 1 } ) - f (\eta) $ and $\eta^{ j , j + 1 } $ denotes the configuration after a particle jumps from a site $j $ to $j + 1 $ if there exists at least one particle on the site $j$:  
\[
\eta^{ j , j + 1 }_k =
\begin{cases}
\begin{aligned}
& \eta_j -1 && \text{ if } k= j , \\
& \eta_{ j + 1} + 1 && \text{ if } k= j + 1 , \\
& \eta_k &&\text{ otherwise.}
\end{aligned}
\end{cases}
\]
The factor $n^2$ is needed to obtain non trivial limit under diffusive scaling. In \cite{gonccalves2015stochastic}, the stochastic Burgers equation is derived in weakly asymmetric regime where $q_n -p_n = O(n^{-1/2})$ as $n$ tends to infinity. Instead, we consider totally asymmetric regime where $q_n - p_n = O(1)$ assuming also the jump rate function $c$ depends on $n$ in an appropriate manner. To simplify the notation, we set $p_n = 0$ and $q_n =1$ in the sequel. Moreover, let $g $ be a positive function on $\mathbb{R}_+$ satisfying the following condition. 

\begin{assumption}
\label{ass:regularity}
Assume the function $g \in C^4_b ( \mathbb{R}_+ : \mathbb{R}_+ ) $ is strictly increasing where $C^4_b$ denotes the family of $C^4$-smooth functions whose all derivatives are bounded, and satisfies $g (0) = 0 $ and $g^\prime(0) > 0$. \end{assumption}

For the function $g$ satisfying Assumption \ref{ass:regularity}, we write $g_n (k ) = n^{ 1/2  } g (n^{ - 1/2 } k ) $ for each $k \in \mathbb{Z}_+ $ and consider the zero-range process with jump rate $c = g_n$ in the above. In other words, we define an operator $L_n$ acting on each local fucntion $f : \mathscr{X} \to \mathbb{R}$ by 
\[
L_n f (\eta ) = n^{ 2 }  \sum_{ j \in \mathbb{Z} } g_n ( \eta_j ) \nabla_{ j , j + 1 } f  (\eta ) ,
\]
and hereafter we consider a Markov process $\eta^n = \{ \eta^n (t ) : t \ge 0 \} $ on $\mathscr{X} $ with infinitesimal generator $L_n $. See Section 2.6 in \cite{kipnis1998scaling} about construction of zero-range processes on infinite volume space where monotonicity of jump rate is postulated. For any probability measure $\mu $ on $\mathscr{X} $, let $\mathbb{P}^n_\mu $ be the distribution of $\eta^n $ on $D ( \mathbb{R}_+ : \mathscr{X} ) $ starting form the initial distribution $\mu $ where $D ( \mathbb{R}_+ : \mathscr{X} ) $ denotes the space of right-continuous processes with left-limits taking values in $\mathscr{X} $ endowed with the Skorohod topology. 

Next we prepare a family of invariant measures of the process $\eta^n $ which are parametrized by density. First for each $\alpha > 0 $, let $\overline{\nu}_\alpha $ be a probability measure on $\mathscr{X} $ whose common marginal is given by 
\[
\overline{\nu}_\alpha (\eta_j = k ) = \frac{1}{ Z_n(\alpha) } \frac{ \alpha^k }{ g_n!(k )}    , \quad 
Z_n (\alpha) = \sum_{ k \ge 0 } \frac{ \alpha^k }{ g_n!(k) }   
\]
where we defined $g_n ! (k ) = g_n (k) \cdots g_n (1 ) $ for each $k \in \mathbb{N}$ and $g_n ! (0) = 1 $. Let $\alpha^*_n$ be the radius of convergence of the partition function $Z_n(\alpha)$. 

\begin{assumption}
\label{ass:density}
Assume that $Z_n(\alpha)$ diverges as $\alpha$ converges to $\alpha^*_n$ for each $n$. 
\end{assumption}

For each $\rho > 0 $, we choose $\Phi_n = \Phi_n (\rho ) $ so that $E_{ \overline{\nu}_{ \Phi_n } } [ \eta_j ] = \rho $ for each $j \in \mathbb{Z}$. This is possible according to Assumption \ref{ass:density}. In this case we have $\Phi_n (\rho ) = E_{\nu^n_\rho}[g_n (\eta )] $ and it is easily verified that 
\begin{equation}
\label{eq:Phi}
\lim_{n \to \infty } \Phi_n (\rho) = g^\prime (0) \rho .
\end{equation}
An example of function $g$ satisfying Assumptions \ref{ass:regularity} and \ref{ass:density} will be given in subsection \ref{subsec:qtasep}.
Hereafter we simply write $\nu^n_\rho = \overline{\nu}_{ \Phi_n (\rho)} $. %and denote variance with respect to $\nu^n_\rho$ by $\chi_n = \mathrm{Var}_{\nu^n_\rho } [\eta_j ] $. 
Then, it is straightforward that the measure $\nu^n_\rho$ satisfies the detailed balance condition and thus $\nu^n_\rho$ is invariant for the process $\eta^n $. To be concerned with equilibrium fluctuations, we only consider the situation when the process starts form these invariant measures. We write $\mathbb{P}_n = \mathbb{P}^n_{ \nu^n_\rho } $ and write the expectation with respect to $\mathbb{P}_n $ by $\mathbb{E}_n $. 

Now we state our main result. For any given constant $ T > 0$, we define density fluctuation field $\{ \mathcal{X}^n_t : t \in [0, T ] \} $ with values on $D ([0, T ], \mathcal{S}^\prime (\mathbb{R} ) ) $ whose action on any test function $\varphi \in \mathcal{S} (\mathbb{R} ) $ is given by  
\begin{equation}
\label{fluctuation}
\mathcal{X}^{  n }_t (\varphi ) = \frac{1}{ \sqrt{n } } \sum_{ j \in \mathbb{Z} } ( \eta^{ n }_j (t ) - \rho ) \varphi \bigg( \frac{ j - f_n t }{ n }  \bigg)   
\end{equation}
where $f_n = f_n (g, \rho) = b_2 n^2 + b_1 n^{ 3 /2  } + b_0 n  $ with 
\begin{equation}
\label{eq:framing}
\begin{aligned}
& b_2 = g^{ \prime } ( 0 ) , \quad 
b_1 = \frac{ 1  }{ 2 } g^{ \prime \prime } (0) ( 1 + 2 \Phi_n (\rho) ) ,  \\
%b_0 = \frac{ 1 }{ 6 } g^{ (3) } (0)  - \frac{ g^{ \prime \prime } (0)^2 }{ g^\prime (0) } ( \Phi_n (\rho )^2 + \Phi_n ( \rho) )  \\
& b_0 =  \frac{g^{(3)}(0)}{6g^\prime(0)}
(1 + 6\Phi_n(\rho) + 3\Phi_n(\rho)^2) 
- \frac{ g^{ \prime \prime } (0)^2 }{ 4g^\prime(0)^2 } 
(1 + 10 \Phi_n(\rho) + 9\Phi_n(\rho)^2) .
\end{aligned}
\end{equation}
Here $\rho = E_{ \nu_\rho } [ \eta (j ) ]$ stands for the density which is conserved for each process $\eta^n = \{ \eta^n (t) : t \ge 0 \}$. The main result of this paper is the following.

\begin{theorem}
\label{mainthm}
Let $\mathcal{X}^n_t $ be the density fluctuation field defined by \eqref{fluctuation} for each zero-range process $\eta^n = \{ \eta^n_t  : t\in [0, T ] \} $. Then the process $\{ \mathcal{X}^n_t : t \in [0, T] \} $ which takes values in $D ([0, T ] , \mathcal{S}^\prime (\mathbb{R})) $ converges in distribution to a unique stationary energy solution $\{ u ( t, \cdot ) : t \in [0, T ] \} $ of the stochastic Burgers equation
\begin{equation}
\label{SBEthm}
\partial_t u =  \frac{ 1 }{ 2 }  g^\prime (0 ) \partial_x^2 u - \frac{ 1 }{2} g^{\prime \prime } (0 ) \partial_x u^2 + \sqrt{ g^\prime ( 0 ) \rho } \partial_x \dot{W} (t, x )  .
\end{equation}
\end{theorem}

%As we stated in the introduction, our dynamics contains an important class of interacting particle system called $q$-TAZRP, which is known to be a dual process of $q$-TASEP (see \cite{borodin2014duality}). We end this section by recalling the $q$-TAZRPs and applying our result for them. 

\subsection{Relation to \texorpdfstring{$q$}{q}-TASEP}
\label{subsec:qtasep}
As we mentioned before, our zero-range process is related in some way to a $q$-deformed version of totally asymmetric simple exclusion process ($q$-TASEP), which is originally introduced in \cite{borodin2014macdonald}. The dynamics of $q$-TASEP is described as follows. Fix a parameter $q \in [ 0, 1 )$. Let $\mathscr{X}_0 = \{ 0, 1 \}^{ \mathbb{Z} } $ be a configuration space with exclusion constraint where similarly to zero-range process we denote each element in $\mathscr{X}_0 $ by Greek letters like $\xi = \{ \xi_j : j \in \mathbb{Z} \} \in \mathscr{X}_0 $. Here $\xi_j$ denotes the occupation number on a site $j$: at each site at most one particle can exist and there is a particle on the site $j$ if $\xi_j = 1$ while there is no particle on that site if $\xi_j = 0$. Then $q$-TASEP is a process which takes values on $\mathscr{X}_0 $ whose infinitesimal generator is given by 
\[
\mathcal{L}_q f (\xi ) = \sum_{ j \in \mathbb{Z} } (1 - q^{ \xi_j } ) \nabla_{ j , j + 1 } f (\xi)  
\]
for each real valued function $f $ on $\mathscr{X}_0 $. Here $\nabla_{ j , j + 1 } f (\xi) = f (\xi^{ j , j + 1 } ) - f (\xi )  $ and $\xi^{ j , j+ 1 } $ denotes the configuration after a particle on site $j $ jumps to the site $ j + 1 $ if possible, similarly to the zero-range case. Note that $q$-TASEP is a kind of exclusion process, which can also be interpreted as zero-range process by the following way. Indeed, let $X (t) = \{ (\cdots, x_0(t) , x_1(t), \cdots) \in \mathbb{Z}_+^{\mathbb{Z} } : x_j (t) \le x_{j+1} (t), j \in \mathbb{Z} \} $ denotes a family of site positions on which particles exist at time $t$. Note that sites except for $\{ x_j(t) : j \in \mathbb{Z} \} $ are empty and we let $\eta_j (t ) = x_{j} (t) - x_{j-1} (t) $ denotes the number of empty sites between $x_j(t) $ and $x_{j+1} (t) $. Then, after taking $q = q_n = \exp ( - n^{- 1/2 } ) $, the process $\eta^n (t)  = \{ \eta_j (n^{5/2} t) : j \in \mathbb{Z} \}$ ($t \ge 0$) which takes values on $\mathscr{X}$ is a zero-range process with generator $L_n $ with $g (x ) = 1- e^{ - x } $. In this case, product $q$-geometric distribution whose marginal distribution is given by 
\[
\overline{ \nu }^n_\alpha (\eta_j = k ) = ( \alpha/\sqrt{n} ;q_n)_\infty \frac{ (\alpha/\sqrt{n} )^k }{ (q_n ; q_n)_k } , \quad 
k \in \mathbb{Z}_+ 
\]
is invariant for the dynamics. Here $(a ; q )_\infty = \prod_{ k = 0 }^\infty (1 - a q^k ) $ and $ ( a ; q )_n = (a ;q)_\infty /(aq^n;q)_\infty $ are the $q$-Pochhammer symbols. Note that the measure $\nu_\alpha $ is a probability measure by the $q$-binomial theorem 
\[
\sum_{ k = 0 }^\infty x^k \frac{(a; q )_k }{ (q ;q )_k} = 
\frac{(ax; q )_\infty}{(x; q)_\infty}
\]
for all $| x | < 1$ and $| q | < 1 $ with $a = 0 $. Moreover, one can easily check that the function $g$ satisfies Assumptions \ref{ass:regularity} and \ref{ass:density}, and $\overline{\nu}^n_\alpha $ converges in law to product Poisson distribution with parameter $\alpha$ when $n$ tends to infinity, as it is expected. 

\section{Proof outline} 
\label{sec:outline}
In this section, we give an outline to the proof of our main theorem (Theorem \ref{mainthm}). Recall the definition of the fluctuation fields $\{ \mathcal{X}^n_t  : t \in [0,  T ] \} $ defined in \eqref{fluctuation}. We write $\varphi^n_j  = \varphi^n_j (t) = \varphi ( ( j - f_n t ) /  n  ) $ for simplicity and define discrete derivative operators $\nabla^n $ and $\Delta^n $ by 
\begin{equation}
\label{eq:discder}
\nabla^n \varphi^n_j = \frac{n}{2}  (\varphi^n_{ j + 1 } - \varphi^n_{j-1} )  , \quad
\Delta^n \varphi^n_j = n^2  (\varphi^n_{ j + 1 } + \varphi^n_{ j -1 } - 2 \varphi^n_j ) . 
\end{equation}
Moreover, let $L^{ *  }_n $ be the $L^2 (\nu_\rho ) $-adjoint operator of $L_n $, which acts on each $f : \mathscr{X} \to \mathbb{R} $ as   
\[
L^{ * }_n f (\eta ) = n^2 \sum_{ j \in \mathbb{Z} } g_n (\eta_j ) \nabla_{ j , j -1 } f (\eta ) .
\]
Then we define symmetric and anti-symmetric part of the operator $L_n $ by $S_n = (L_n + L^*_n )/ n $ and $A_n = (L_n - L^*_n ) / 2 $, respectively. We begin with a martingale decomposition associated Markov process. By Dynkin's martingale formula, for each test function $\varphi \in \mathcal{S} (\mathbb{R}) $,
\begin{equation}
\label{eq:mart}
\mathcal{M}^n_t (\varphi ) = \mathcal{X}^n_t (\varphi) - \mathcal{X}^n_0 (\varphi) 
- \int_0^t (\partial_s + L_n ) \mathcal{X}^n_s (\varphi) ds 
\end{equation}
is a mean-zero martingale with quadratic variation 
\begin{equation}
\label{qv}
\begin{aligned}
\langle \mathcal{M}^n (\varphi ) \rangle_t 
& =\int_0^t \big( L_n ( \mathcal{X}^n_s (\varphi ) )^2 - 2  \mathcal{X}^n_s (\varphi ) L_n \mathcal{X}^n_s (\varphi ) \big) ds \\
& = \frac{1}{n} \int_0^t \sum_{j \in \mathbb{Z} } g_n (\eta^n_j (s) ) \big[ n ( \varphi^n_{j+1} (s) - \varphi^n_j (s) ) \big]^2 ds ,
\end{aligned}
\end{equation}
which is expected to converge to $g^\prime(0) \rho \| \partial_x \varphi \|^2_{L^2 (\mathbb{R} ) } t $ as $n$ tends to infinity in view of \eqref{eq:Phi}. Next, to decompose the additive functional into symmetric and anti-symmetric parts, we write $L_n = S_n + A_n $ and split the third term of right-hand side of \eqref{eq:mart} into sum of 
\begin{equation}
\label{eq:symm}
\begin{aligned}
\mathcal{S}^n_t (\varphi) 
= \int_0^t S_n \mathcal{X}^n_s (\varphi ) ds  
= \frac{1}{ 2\sqrt{n} } \int_0^t \sum_{ j \in \mathbb{Z} } g_n (\eta^n_j (s) ) \Delta^n \varphi^n_j (s) ds
\end{aligned}
\end{equation}
and
\begin{equation}
\label{eq:anti-symm}
\begin{aligned}
\mathcal{B}^n_t (\varphi)
& = \int_0^t ( \partial_s + A_n ) \mathcal{X}^n_s (\varphi )  ds \\
& = \frac{1}{ \sqrt{n} } \int_0^t \sum_{ j \in \mathbb{Z} } \big[ n g_n (\eta^n_j (s) ) \nabla^n \varphi^n_j (s) - \frac{f_n}{n} (\eta^n_j (s) - \rho ) \partial_x \varphi^n_j  (s) \big] ds  . 
\end{aligned}
\end{equation}
By this line we obtain a decomposition 
\begin{equation}
\label{decomposition}
\mathcal{X}^n_t (\varphi ) =  \mathcal{X}^n_0 (\varphi ) + \mathcal{S}^n_t (\varphi ) + \mathcal{B}^n_t (\varphi ) + \mathcal{M}^n_t (\varphi ) .
\end{equation}
In the sequel, we show tightness of each term in the decomposition \eqref{decomposition} and characterize limiting points. Before consider in detail, let us roughly see the convergence of these terms. For that purpose, a Taylor expansion of $g_n $ in occupation variables 
\begin{equation}
\label{Taylor}
g_n (k) = n^{  1/ 2 } g (n^{ - 1/ 2} k ) 
= g^\prime (0) k 
+ \frac{ 1 }{ 2} g^{\prime \prime } (0 ) k^2 n^{-1/2} 
%+ \frac{ 1  }{ 6 } g^{ (3) } (0) k^3 n^{ - 1 }  
 + O (n^{ - 1  } ) 
\end{equation}
plays an important role. We begin with the symmetric part. Taking only the leading term of $g_n $, the process $\mathcal{S}^n_t (\varphi ) $ can be replaced by $(g^\prime (0)/2) \mathcal{X}^n_t (\Delta^n \varphi ) $ and thus one can show that this term converges to the viscosity term in the equation \eqref{SBE} tested against each function $\varphi $. Next, for the anti-symmetric part, we see that the continuous derivative can be replaced by the discrete one. Indeed, let
\[
E^n_t (\varphi) 
= \frac{1}{\sqrt{n}} \int_0^t \sum_{j\in \mathbb{Z}}
\frac{f_n}{n} (\eta^n_j (s)-\rho) (\partial_x \varphi^n_j (s) - \nabla^n \varphi^n_j (s)) ds .
\]
Recall the definition of the discrete derivative given in \eqref{eq:discder}. Then the mean-value theorem yields $|\partial_x \varphi^n_j (t) - \nabla^n \varphi^n_j (s) |= O(n^{-2}) $. Hence by the Schwarz inequality, we have that
\[
\mathbb{E}_n \bigg[\sup_{0 \le t \le T} |E^n_t (\varphi)|^2 \bigg]
\le T \frac{f_n^2}{n^3} \int_0^T \sum_{j \in \mathbb{Z}}
\mathbb{E}_n \big[ (\eta^n_j(s)-\rho)^2 \big] 
(\partial_x \varphi^n_j (s)-\nabla^n \varphi^n_j (s))^2 ds 
\le C\frac{T^2}{n^2} . 
\]
Now we expand $ n g_n(\eta_j) - n^{-1}f_n(\eta_j -\rho ) $ in occupation variable $\eta_j $ with the help of the expansion \eqref{Taylor}. Then, choosing the framing $f_n$ carefully, the linear terms cancel and one can expect the leading part has order two. Then $\mathcal{B}^n_t (\varphi ) $ can roughly be written by a quadratic functional of the fluctuation field, which give rise to the non-linear term in the limiting equation \eqref{SBEthm}. This is clarified in the following manner. Let us define random variables $W_j$ and its centered version $\overline{W}_j$ by 
\begin{equation}
\label{eq:wdef}
W_j (t) = g_n (\eta^n_j (t ) ) / g^\prime (0) , \quad 
\overline{W }_j (t)= W_j (t) - \Phi_n (\rho ) / g^\prime (0) .
\end{equation}
Moreover, we define a modified process $\tilde{\mathcal{B} }^n_\cdot \in D ( [0, T ], \mathcal{S}^\prime (\mathbb{R} ) ) $ by
\begin{equation}
\label{modified}
\tilde{ \mathcal{B} }^n_t ( \varphi ) = \frac{ g^{ \prime \prime } (0) }{ 2 } \int_0^t \sum_{ j \in \mathbb{Z} } \overline{W}_{ j -1 } (s) \overline{W}_j (s) \nabla^n \varphi^n_j (s) ds  
\end{equation}
for each test function $\varphi \in \mathcal{S} (\mathbb{R}) $. Then we have the following result. 

\begin{lemma}
\label{antisymm}
We have that 
\[
\limsup_{ n \to \infty } \mathbb{E}_n \bigg[  \sup_{ 0 \le t \le T } 
\big| \mathcal{B}^n_t (\varphi ) - \tilde{ \mathcal{B} }^n_t (\varphi ) \big|^2 \bigg] = 0 .
\]
\end{lemma}
Lemma \ref{antisymm} is a consequence of an expansion of $W_j - \eta_j $ (see \eqref{expansion}), which can be deduced by a similar way as for the O'Connell-Yor polymer model given in \cite{jara2020stationary}. The proof of Lemma \ref{antisymm} is postponed to Section \ref{sec:expansion}.

%%%%%%%%%%%%%%%%%%%%%%%%%%%%%%%%%%%%%%%%%%%%%%
%   Second-order Boltzmann-Gibbs principle   % 
%%%%%%%%%%%%%%%%%%%%%%%%%%%%%%%%%%%%%%%%%%%%%%
\section{The second-order Boltzmann-Gibbs principle}
\label{sec:BG2}
For each $\ell \in \mathbb{N}$ we denote a centered local average of $W_j $'s by $\overrightarrow{W}^\ell_j = \ell^{ - 1 } \sum_{ k = 0, \ldots, \ell -1  } \overline{W}_{ j + k } $. In addition, let $\tau_j $ denotes the canonical shift: $\tau_j \eta_i = \eta_{ i + j } $. We define 
\begin{equation}
\label{Q}
\mathcal{Q}^n_\rho (\ell ; t ) =  \frac{ g^{ \prime \prime } (0) }{ 2 } \bigg( \big( \overrightarrow{W}^\ell_0 (t) \big)^2 - \frac{ \sigma^2_n (\rho ) }{ \ell } \bigg) 
\end{equation}
where $\sigma_n^2 (\rho) = \mathrm{Var}_{\nu^n_\rho}[W_0]$. Then the following result is central to demonstrate the main theorem. 

\begin{theorem}[Second-order Boltzmann-Gibbs principle]
\label{BG2}
We have that 
\[
\begin{aligned}
\mathbb{E}_n \bigg[  \sup_{ 0 \le t \le T } \bigg| \tilde{ \mathcal{B} }^n_t (\varphi ) - \int_0^t \sum_{ j \in \mathbb{Z} } \tau_j \mathcal{Q}^n_\rho (\ell; s ) \nabla^n \varphi^n_j (s) ds \bigg|^2 \bigg] 
\le C \bigg( \frac{ \ell }{ n^2 } + \frac{ T }{ \ell^{  2 } } \bigg) \int_0^T \sum_{ j \in \mathbb{Z} } (\nabla^n \varphi^n_j ( t ) )^2 dt .   
\end{aligned}
\]
\end{theorem}

\subsection{Preliminaries}
To give a proof Theorem \ref{BG2}, we prepare for some basic tools. For each local $L^2 (\nu^n_\rho ) $ function $F$, we define its $H^{1, n } $-norm by $\| F \|^2_{ 1 , n } = \langle F , - S_n F \rangle_{ L^2 (\nu^n_\rho ) } $, which is explicitly represented as
\[
\| F \|^2_{1 , n }
= \frac{ n^2}{2} \sum_{ j \in \mathbb{Z}, \, | j - j^\prime | = 1 } E_{ \nu^n_\rho } [g_n (\eta_j ) (\nabla_{ j ,j^\prime } F (\eta))^2 ] . 
\]
Moreover we define the $H^{-1 , n } $-norm through the variational formula 
\[
\begin{aligned}
 \| F \|^2_{- 1, n } 
  = \sup_{ f \in L^2 (\nu^n_\rho ), \, \text{local} } \big\{ 2 \langle F , f \rangle_{ L^2 (\nu^n_\rho ) } - \| f \|^2_{1, n } \big\} .
\end{aligned}
\]
Then the following Kipnis-Varadhan inequality holds true. 

\begin{proposition}[Kipnis-Varadhan inequality]
\label{KV}
Let $F : [0, T ] \to L^2 (\nu_\rho ) $ be a function such that $E_{\nu_\rho} [F(t,\cdot ) ] = 0 $ for each $t \in [0, T]$. Then there exists a positive constant $C$ such that 
\[
\mathbb{E}_n \bigg[ \sup_{ 0 \le t \le T } \bigg| \int_0^t F (s, \eta^n (s) ) ds \bigg|^2 \bigg]
\le C \int_0^T \| F(t, \cdot ) \|^2_{ - 1, n } dt . 
\]
\end{proposition}

Moreover, the following integration-by-parts formula is a key ingredient.

\begin{lemma}[Integration by parts]
\label{IBP}
Let $f $ be any real-valued local $L^2 (\nu^n_\rho ) $-function on $\mathscr{X} $. Then for each $j \in \mathbb{Z} $ we have an identity 
\[
E_{ \nu^n_\rho } [ f ( \eta ) (  W_j  - W_{ j + 1 }  ) ] = - E_{ \nu^n_\rho } [\nabla_{ j , j + 1 } f (\eta ) W_j ] .
\]
\end{lemma}
\begin{proof}
First we observe the invariant measure $\nu^n_\rho $ satisfies an identity 
\[
\nu^n_\rho (\eta^{ j , j + 1 } ) = \frac{  g_n ( \eta_j ) }{ g_n (\eta_{ j + 1}  + 1 )  } \nu^n_\rho (\eta) 
\]
for each $j \in \mathbb{Z} $ and each configuration $\eta \in \mathscr{X} $ such that $\eta_j > 0 $. From this identity, for any local function $f$ we see that a change of variables yields  
\[
\begin{aligned}
E_{ \nu^n_\rho } [ f (\eta) g_n (\eta_{ j + 1 }) ] 
 = \sum_{\eta \in \mathscr{X} } \sum_{ j \in \mathbb{Z} }  f (\eta^{ j , j  + 1 } ) g_n (\eta_{ j +  1 } + 1  ) \frac{ g_n (\eta_j ) }{ g_n (\eta_{j + 1}  + 1 )  } \nu^n_\rho (\eta) = E_{ \nu^n_\rho} [ f (\eta^{ j , j + 1 } ) g_n (\eta_j ) ] . 
\end{aligned}
\]
Now subtract $E_{ \nu^n_\rho } [ f (\eta) g_n (\eta_j) ] $ from both sides to complete the proof. 
\end{proof}

\subsection{Proof of Theorem \ref{BG2}} 
%In the sequel, we often drop the dependency on scale parameter and time. 
To show Theorem \ref{BG2}, we use a decomposition  
\[
\begin{aligned}
\frac{g^{\prime \prime}(0)}{2} \overline{ W }_{j - 1 } \overline{ W }_j  - \tau_j \mathcal{Q}^n_\rho (\ell )  
 = \frac{ g^{ \prime \prime } (0) }{ 2 } \bigg( \overline{ W }_{ j - 1 } ( \overline{ W}_j  - \overrightarrow{W}^\ell_j  ) 
 + \overrightarrow{ W }^\ell_j ( \overline{W}_{ j  - 1 }  - \overrightarrow{W}^\ell_j  ) + \frac{ \sigma^2_n ( \rho ) }{ \ell } \bigg) .
\end{aligned}
\]
Then following Lemmas \ref{oneblock} and \ref{BGmain} finish the proof of the second-order Boltzmann-Gibbs principle.

\begin{lemma}
\label{oneblock}
There exists a positive constant $C$ such that  
\[
\begin{aligned}
 \mathbb{E}_n \bigg[ \sup_{ 0 \le t \le T } \bigg| \int_0^t \sum_{ j \in \mathbb{Z} } \overline{ W }_{ j -1 } ( s ) ( \overline{W}_j ( s ) - \overrightarrow{W}^\ell_j ( s ) ) \varphi_j (s) ds \bigg|^2 \bigg] 
 \le C \frac{ \ell }{ n^2   }  \int_0^T \sum_{ j \in \mathbb{Z} } ( \varphi_j (t ) )^2 dt .
\end{aligned}
\]
\end{lemma}
\begin{proof}
Recalling the definition of the local average, we first observe 
\[
\begin{aligned}
\sum_{  j \in \mathbb{Z} } \overline{ W }_{ j -1 } ( \overline{W}_j - \overrightarrow{W}^\ell_j  ) \varphi_j 
& = \sum_{ j \in \mathbb{Z} } \overline{ W }_{ j -1 } \sum_{ i = 0 }^{ \ell -2  } ( W_{ j + i } - W_{ j + i  +1 } ) \psi_{ i + 1} \varphi_j  \\
& = \sum_{ k \in \mathbb{Z} } F_k ( W_k - W_{k + 1} ) 
\end{aligned}
\]
where $\psi_i = ( \ell - i ) / \ell $ and $F_k = \sum_{ i = 0 , \ldots, \ell - 2 } \overline{ W }_{ k - i - 1 } \psi_{i + 1 } \varphi_{ k - i } $, and in the second identity we let $k = j + i $ to rearrange the sum. Here we note that the local function $F_j $ is invariant under the action $\sigma_{ j , j + 1 } $. We fix any local, $L^2 (\nu^n_\rho ) $-function $f $. Then, according to the integration by parts formula (Lemma \ref{IBP}), we have   
\[
\begin{aligned}
 2 \bigg\langle \sum_{ j \in \mathbb{Z} } \overline{ W }_{ j - 1 } ( \overline{W}_j  - \overrightarrow{W}^\ell_j  ) \varphi_j , f \bigg\rangle_{ L^2 (\nu^n_\rho ) } 
&= 2 \sum_{ j \in \mathbb{Z} } E_{ \nu^n_\rho } [ F_j ( W_j  - W_{ j + 1 } )  f  ] \\
& = - 2 \sum_{ j \in \mathbb{Z} } E_{ \nu^n_\rho } [ W_j (\nabla_{ j , j + 1 } f ) F_j  ] . 
\end{aligned}
\]  
By Young's inequality, we notice that the last display can be absolutely bounded by 
\[
\frac{n^2}{2}  \sum_{ j \in \mathbb{Z} } E_{\nu^n_\rho } [ g_n (\eta_j ) (\nabla_{ j , j + 1 } f (\eta) )^2 ] 
+ \frac{2}{ n^2 g^\prime (0)^2 }  \sum_{ j \in \mathbb{Z} } E_{ \nu^n_\rho } [ g_n (\eta_j ) ] E_{\nu^n_\rho } [ F_j^2] 
\]
where for the second term we noted that $F_j $ is independent of $\eta_j $. Moreover, note that the first term is bounded by $\| f \|^2_{1,n}$. Then, since $\sum_{ j } E_{ \nu_\rho } [F^2_j ] \le C \ell \sum_{j } \varphi^2_j $ we conclude with the help of the Kipnis-Varadhan inequality (Proposition \ref{KV}) that 
\[
\mathbb{E}_n \bigg[ \sup_{ 0 \le t \le T } \bigg| \int_0^t \sum_{ j \in \mathbb{Z} } 
\big[ \overline{ W }_{ j -1 } ( s) ( \overline{ W }_j (s ) - \overrightarrow{W}^\ell_j (s) ) \varphi_j (s)  \big] ds \bigg|^2 \bigg] 
 \le C \frac{ \ell }{ n^2 } \int_0^T \sum_{ j \in \mathbb{Z} } ( \varphi_j (t ) )^2 dt 
\]
for some $C > 0 $. Hence we complete the proof.  
\end{proof}

\begin{lemma}
\label{BGmain}
There exists a positive constant $C$ such that 
\[
\begin{aligned}
&  \mathbb{E}_n \bigg[ \sup_{ 0 \le t \le T } \bigg| \int_0^t \sum_{ j \in \mathbb{Z} } 
\big[ \overrightarrow{ W }^\ell_j (s) ( \overline{W}_{j -1} (s) - \overrightarrow{W}^\ell_j (s ) ) + \frac{ \sigma^2_n ( \rho ) }{ \ell }  \big]  \varphi_j (s) ds \bigg|^2 \bigg] \\
& \quad \le C \bigg( \frac{ \ell }{ n^2 }  + \frac{ T }{\ell^2  } \bigg) \int_0^T  \sum_{ j \in \mathbb{Z} } \varphi_j (t )^2 dt .
\end{aligned}
\]
\end{lemma}
\begin{proof}
Recalling the definition of local averages, we have 
\[
\begin{aligned}
\sum_{  j \in \mathbb{Z} } \overrightarrow{ W }^\ell_j ( \overline{W}_{j - 1 } - \overrightarrow{ W }^\ell_j ) \varphi_j 
 = \sum_{ j \in \mathbb{Z} } \overrightarrow{ W }^\ell_j  \sum_{ i = 0 }^{ \ell - 1 } ( \overline{W}_{ j + i - 1 } - \overline{W}_{ j + i } ) \psi_i \varphi_j  
\end{aligned}
\]
where $\psi_i = ( \ell - i ) / \ell $. Then we fix any local function $f : \mathscr{X} \to \mathbb{R} $ and apply integration by parts (Lemma \ref{IBP}). 

First we consider the case $1 \le i \le \ell - 1 $. To make notations simple, we write $W_j^+ = g_n (\eta_j + 1 ) / g^\prime (0) $ and $W^-_j = g_n (\eta_j -1 )/ g^\prime (0) $. Then noting 
\[
\begin{aligned}
\nabla_{ j + i -1 , j + i } \overrightarrow{ W }^\ell_j  
 =  \ell^{ - 1 } ( W^-_{ j + i -1 }  - W_{ j + i -1} + W^+_{ j + i } - W_{ j + i } )
\end{aligned}
\]
provided $\eta_{ j + i - 1 } > 0$, we have 
\[
\begin{aligned}
& E_{ \nu^n_\rho } \big[ \overrightarrow{ W }^\ell_j ( \overline{W}_{ j + i - 1}  - \overline{W}_{ j + i } ) f  \big] \\
&\quad  = - E_{ \nu^n_\rho } [ (\sigma_{ j + i - 1 , j + i  }  \overrightarrow{ W }^\ell_j  ) ( \nabla_{ j + i - 1 , j + i  } f  ) W_{ j + i  -1 } ]
- E_{ \nu^n_\rho } [ f ( \nabla_{ j + i -1 , j + i }  \overrightarrow{ W }^\ell_j ) W_{ j + i - 1 } ]  .
\end{aligned}
\]
For the first term in the right-hand side, we can use Young's inequality and then apply the Kipnis-Varadhan inequality to get a bound of order $\ell /n^2 $ by the same calculation as in Lemma \ref{oneblock}. For the second term, to make use a similar procedure, we decompose   
\[
\begin{aligned}
& W_{j + i - 1 } ( W^-_{ j + i -1 }  - W_{ j + i -1} + W^+_{ j + i } - W_{ j + i } ) \\
& \quad = - ( W_{ j + i -1 } W_{ j + i }  - W_{ j + i -1 }  W^-_{ j + i - 1  } )  
+ ( W_{ j + i -1 } W^+_{ j + i } - W_{ j + i }^2 ) 
+ ( W_{ j + i  } ^2 - W_{ j + i - 1 }^2 ) .
\end{aligned}
\]
Then one can find that 
\[
\begin{aligned}
& E_{ \nu^n_\rho } [ f ( \nabla_{ j + i -1 , j + i }  \overrightarrow{ W }^\ell_j ) W_{ j + i - 1 } ] \\
& \quad = - \ell^{ -1 } E_{ \nu^n_\rho } [ W_{ j + i - 1 } W^-_{ j + i - 1} ( \nabla_{ j + i - 1, j + i } f ) ] 
+ \ell^{ -1 } E_{ \nu^n_\rho } [ W_{ j + i }^2 (\nabla_{ j + i, j + i - 1 } f ) ]  \\
& \qquad  + \ell^{ - 1 } E_{ \nu^n_\rho } [ f (W_{ j + i }^2 - W_{ j + i - 1 }^2 ) ] .
\end{aligned}
\]
From the first two terms we deduce the $\ell /n^2 $-bound with the help of the Kipnis-Varadhan inequality and thus hereafter we may only consider the third term. 

On the other hand, we consider the case $i = 0 $. Similarly to the above, a Leibniz rule for the derivative operator $\nabla_{j-1, j}$ yields  
\begin{equation*}
\begin{aligned}
& E_{ \nu^n_\rho } \big[ \overrightarrow{ W }^\ell_j ( \overline{W}_{ j - 1} - \overline{W}_j ) f \big] \\
& \quad  = - E_{ \nu^n_\rho } [ (\sigma_{ j - 1 , j }  \overrightarrow{ W }^\ell_j ) ( \nabla_{ j  -1 , j  } f  ) W_{ j  -1 } ]
- E_{ \nu^n_\rho } [ f ( \nabla_{ j -1 , j }  \overrightarrow{ W }^\ell_j  ) W_{ j  - 1 } ] .
\end{aligned}
\end{equation*}
Then the first term in the last display gives an $\ell /n^2$-bound so that we may focus on the second term. Moreover, note that 
\[
\begin{aligned}
\nabla_{ j -1 , j } \overrightarrow{ W }^\ell_j 
  = \ell^{ - 1 } ( W^+_j   - W_j )    
\end{aligned}
\]
provided $\eta_{j -1 } > 0 $. To take advantage of a similar calculation as the case $i \ge 1$, we use a decomposition 
\[
\begin{aligned}
 W_{ j - 1 } ( W^+_j - W_j ) 
 = (W_{j-1} W_j^+ - W_j^2) + (W_j^2 - W_{j-1} W_j) .
%& = W_{j - 1 } (  W^+_{ j } - W_{ j  } + W^-_{ j -1 }  - W_{ j -1} ) - W_{ j -1 } ( W^-_{ j -1 }  - W_{ j -1} ) \\
%& = - ( W_{ j -1 } W_{ j }  - W_{ j -1 }  W^-_{ j - 1  } )  
%+ ( W_{ j  -1 } W^+_{ j + 1} - W_{ j }^2 ) \\
%& \quad + ( W_{ j  } ^2 - W_{ j - 1 }^2 ) - W_{ j -1 } ( W^-_{ j -1 }  - W_{ j -1} )  .
\end{aligned}
\]
Then we have that
\begin{equation}
\label{i0}
\begin{aligned}
 E_{ \nu^n_\rho } [ f ( \nabla_{ j -1 , j }  \overrightarrow{ W }^\ell_j  ) W_{ j  - 1 } ] 
 = \ell^{-1} E_{\nu^n_\rho}[W_j^2 (\nabla_{j,j-1}f)] 
 + \ell^{-1} E_{\nu^n_\rho} [W_j (W_j -W_{j-1})f] . 
%& \quad  = 
% \ell^{ -1 } E_{ \nu^n_\rho } [ W_{ j - 1 } W^-_{ j - 1} ( \nabla_{ j - 1, j } f ) ] - \ell^{ -1 } E_{ \nu^n_\rho } [ W_{ j }^2 (\nabla_{ j, j - 1 } f ) ]  \\
%& \qquad  - \ell^{ - 1 } E_{ \nu^n_\rho } [ f (W_{ j }^2 - W_{ j - 1 }^2 ) ] + \ell^{ - 1 } E_{ \nu^n_\rho } [  W_{ j -1 } ( W^-_{ j -1 }  - W_{ j -1}) f] .
\end{aligned}
\end{equation}
From the first term, we get an extra $\ell /n^2 $ factor to bound them and thus we only consider the second term. 

By this line, we obtained 
\[
\begin{aligned}
& \mathbb{E}_n \bigg[ \sup_{ 0 \le t \le T } \bigg| \int_0^t \sum_{ j \in \mathbb{Z} } 
\big[ \overrightarrow{ W }^\ell_j (s) ( \overline{W}_{j -1} (s) - \overrightarrow{W}^\ell_j (s ) ) + \frac{ \sigma^2_n ( \rho ) }{ \ell }  \big]  \varphi_j (s) ds - \mathcal{R}^n_t (\varphi) \bigg|^2 \bigg] \\
& \quad \le C \frac{ \ell }{ n^2 } \int_0^T  \sum_{ j \in \mathbb{Z} } \varphi_j (t )^2 dt 
\end{aligned}
\]
where 
\[
\mathcal{R}^n_t (\varphi) 
= \frac{1}{\ell} \int_0^t \sum_{j \in \mathbb{Z}} 
\bigg[ - \sum_{i=1}^{\ell-1} (W^2_{j+i}-W^2_{j+i-1})\psi_i 
- W_j (W_j - W_{j-1}) + \sigma_n^2(\rho) \bigg] \varphi^n_j (s )ds .
\]
Now our task is to estimate $\mathcal{R}^n$ directly. Here we notice that  
\[
\sum_{i=1}^{\ell-1} (W^2_{j+i} - W^2_{j+i-1})\psi_i = (\overrightarrow{W^2})^\ell_j 
- ( W^2_j - E_{\nu_\rho}[W^2_j] )
\]
where $(\overrightarrow{ W^2 })^\ell_j = \ell^{ - 1} \sum_{ i = 0, \ldots, \ell- 1 } ( W^2_{j+i} - E_{\nu_\rho}[W^2_{j+i}] ) $ is centered local averages defined similarly for $W_j$. Then by Schwarz's inequality and stationarity, we have 
\[
\begin{aligned}
\mathbb{E}_n  \bigg[ \sup_{ 0 \le t \le T} \bigg| \int_0^t \frac{ 1 }{ \ell } \sum_{ j \in \mathbb{Z} } 
(\overrightarrow{W^2})^\ell_j (s) \varphi_j (s) ds  \bigg|^2   \bigg] 
& \le \frac{ T }{ \ell^2 } \mathbb{E}_n \bigg[ \sup_{ 0 \le t \le T } \int_0^t \bigg( \sum_{ j \in \mathbb{Z} } 
(\overrightarrow{W^2})^\ell_j (s) \varphi_j (s) ds \bigg)^2 \bigg] \\
& \le \frac{ T }{ \ell^2 } \int_0^T \mathbb{E}_n \bigg[ \bigg( \sum_{ j \in \mathbb{Z} } 
(\overrightarrow{W^2})^\ell_j (s) \varphi_j (s) \bigg)^2 \bigg] ds .
\end{aligned}
\]
However, noting $(\overrightarrow{W^2})^\ell_j $ and $(\overrightarrow{W^2})^\ell_k $ are independent if $| j - k | \ge \ell $, we have that  
\[
\begin{aligned}
\mathbb{E}_n \bigg[ \bigg( \sum_{ j \in \mathbb{Z} } (\overrightarrow{W^2})^\ell_j  \varphi_j \bigg)^2 \bigg] 
& = \mathbb{E}_n \bigg[ \bigg( \sum_{ k = 0 }^{ \ell - 1 } \sum_{ j \in \mathbb{Z} } (\overrightarrow{W^2})^\ell_{ \ell j + k }   \varphi_{ \ell j + k } \bigg)^2 \bigg] \\
& \le \ell \sum_{ k = 0 }^{ \ell -1 } \mathbb{E}_n \bigg[ \bigg( \sum_{ j \in \mathbb{Z} } (\overrightarrow{W^2})^\ell_{ \ell j + k }   \varphi_{ \ell j + k } \bigg)^2 \bigg] \\
& = \ell \sum_{ k = 0 }^{ \ell - 1 } \sum_{ j \in \mathbb{Z} } \mathbb{E}_n \big[ \big( (\overrightarrow{W^2})^\ell_{ \ell j + k } \big)^2 \big] \varphi_{ \ell j + k }^2 
\le C \sum_{ k = 0 }^{ \ell -1 } \sum_{ j \in \mathbb{Z} } \varphi_{ \ell j + k }^2  
\end{aligned}
\]
for some $C > 0 $ since $\mathbb{E}_n [ ( (\overrightarrow{W^2})^\ell_j )^2 ] \le C \ell^{ - 1 } $ for each $j $. On the other hand, since random variables $W_j^2-E_{\nu_\rho} [W^2_j]$ and $W_j(W_j - W_{j-1}) - \sigma_n^2(\rho)$ are centered, Schwarz's inequality bring us the desired estimate. Hence combining all the estimates we complete the proof. 
\end{proof}

%%%%%%%%%%%%%%%%%%%%%%%%%%%%%%%%
%                  Tightness                      %
%%%%%%%%%%%%%%%%%%%%%%%%%%%%%%%%
\section{Tightness}
\label{sec:tightness}
%In this section, we show that the sequence $\{ \mathcal{X}^n_t : t\in [0, T \}_{ n \in \mathbb{N}} $ is tight in the uniform topology on $D([0, T] , \mathcal{S}^\prime (\mathbb{R}))$. For that purpose, r
Recall the martingale decomposition \eqref{decomposition}. Note that according to Lemma \ref{antisymm} for the anti-symmetric part, we may consider the modified process $\{ \tilde{ \mathcal{B} }^n_t : t \in [0,T] \} $ defined by \eqref{modified} instead of the original process $\{ \mathcal{B}^n_t : t \in [0,T] \} $. We show tightness of each process as follows. 

\begin{lemma}
\label{tightness}
The sequences $\{ \mathcal{X}^n_t : t \in [0, T ] \}_{ n \in \mathbb{N} } $, $\{ \mathcal{M}^n_t : t \in [0, T ] \}_{ n \in \mathbb{N} } $, $\{ \mathcal{S}^n_t : t \in [0, T ] \}_{ n \in \mathbb{N} } $ and $\{ \tilde{\mathcal{B}}^n_t : t \in [0, T ] \}_{ n \in \mathbb{N} } $, when the processes start from the invariant measure $\nu^n_\rho $, are tight in the uniform topology on $D ([0, T ] , \mathcal{S}^\prime (\mathbb{R} ) ) $. 
\end{lemma}

To prove tightness of a sequence of processes, the following criteria are helpful.

\begin{proposition}[Mitoma's criterion, \cite{mitoma1983tightness}]
\label{Mitoma}
A sequence of $\mathcal{S}^\prime (\mathbb{R} ) $-valued processes $\{ \mathcal{Y}^n_t : t \in [0, T ] \}_{ n \in \mathbb{N} } $ with trajectories in $D ([0, T ] , \mathcal{S}^\prime (\mathbb{R} ) ) $ is tight with respect to the Skorohod topology if and only if the sequence $\{ \mathcal{Y}^n_t (\varphi ) : t \in [0, T ] \}_{ n \in \mathbb{N}  }  $ of real-valued processes is tight with respect to the Skorohod topology of $D ([0, T ] , \mathbb{R} ) $ for any $\varphi \in \mathcal{S} ( \mathbb{R} ) $.   
\end{proposition}

\begin{proposition}[Aldous' criterion]
\label{Aldous}
A sequence $\{ X^n_t : t \in [0, T ] \}_{ n \in \mathbb{N} } $ of real-valued processes is tight with respect to the Skorohod topology of $D ([0, T ] , \mathbb{R} ) $ if the following two conditions hold. 
\begin{enumerate}
\item 
The sequence of real-valued random variables $\{ X^n_t \}_{ n \in \mathbb{N} } $ is tight for any $ t\in [0, T ] $. 

\item 
For any $\varepsilon > 0$, 
\[
\lim_{ \delta \to 0 } \limsup_{  n \to \infty } \sup_{ \gamma \le \delta } \sup_{ \tau \in \mathcal{T}_T } 
\mathbb{P}_n \big( | X^n_{ \tau +  \gamma } - X^n_\tau  | > \varepsilon \big) = 0 
\]
where $\mathcal{T}_T$ is the set of stopping times bounded by $T $ using the convention $X^n_{ \tau + \gamma } = X^n_T$ if $ \tau + \gamma > T$. 
\end{enumerate}
\end{proposition}

The rest of this subsection is devoted to prove Lemma \ref{tightness}. With the help of Mitoma's criterion \cite{mitoma1983tightness}, it suffices to show tightness of sequences $\{ \mathcal{X}^n_t (\varphi) : t \in [0, T ] \}_{ n \in \mathbb{N} } $, $\{ \mathcal{S}^n_t (\varphi) : t \in [0, T ] \}_{ n \in \mathbb{N} } $, $\{ \mathcal{A}^n_t (\varphi) : t \in [0, T ] \}_{ n \in \mathbb{N} } $ and $\{ \mathcal{M}^n_t (\varphi) : t \in [0, T ] \}_{ n \in \mathbb{N} } $ with respect to the uniform topology on $D ([0, T ]  , \mathbb{R} ) $ for any given test function $\varphi \in \mathcal{S} (\mathbb{R} ) $. Moreover, one can notice that the sequence of random variables $\{ \mathcal{X}^n_0 (\varphi ) \}_{ n \in \mathbb{N}} $ converges to a mean-zero normal random variable with variance $\rho \| \varphi \|^2_{L^2(\mathbb{R})}$, which particularly shows that the sequence $\{ \mathcal{X}^n_0 \}_{n \in \mathbb{N} } $ is tight. Hence from here we focus on tightness of martingale, symmetric and anti-symmetric parts. 

%Then tightness of $\{ \mathcal{X}^n_t (\varphi) : t \in [0, T ] \}_{ n \in \mathbb{N} } $ is a direct consequence from tightness of the other terms in the martingale decomposition \eqref{decomposition}. We prove tightness of these three other terms in the sequel. 

\subsection{Martingale part}
First we consider the martingale term. Recall that quadratic variation of the martingale $\mathcal{M}^n_t (\varphi ) $ is given by \eqref{qv}. 
%Then according to the Burkholder-Davis-Gundy inequality, we have 
%\[
%\mathbb{E}_n [| \mathcal{M}^n_{t_2} (\varphi) - \mathcal{M}^n_{t_1}(\varphi) |^p ] 
%\le C | t_2 - t_1 |^{ p /2 } \frac{1}{n} \int^t_0 \sum_{j \in \mathbb{Z} } (\nabla^n %\varphi^n_j (s))^2 ds 
%\]
Then for any stopping time $\tau \in \mathcal{T}_T $ we have 
\[
\begin{aligned}
\mathbb{P}_{ \nu_\rho } \big(  | \mathcal{M}^n_{ \tau + \gamma } (\varphi) - \mathcal{M}^n_\tau (\varphi)  | > \varepsilon \big) 
& \le \varepsilon^{- 2 } \mathbb{E}_{\nu_\rho } \big[ | \mathcal{M}^n_{\tau + \gamma } (\varphi) - \mathcal{M}^n_\tau (\varphi)  |^2  \big] \\
& \le  \varepsilon^{ - 2 } \mathbb{E}_{ \nu_\rho } \bigg[ \int_\tau^{\tau + \gamma } \frac{1}{n} \sum_{ j \in \mathbb{Z} } g_n (\eta^n_j (s)) (\nabla^n \varphi^n_j )^2 (s) ds \bigg] \\
& \le C \varepsilon^{-2} \int_\tau^{\tau + \gamma} \frac{1}{n} \sum_{j \in \mathbb{Z} } (\nabla^n \varphi^n_j (s))^2 ds  
\end{aligned}
\]
since $E_{\nu^n_\rho } [g_n (\eta)] = \Phi_n (\rho) $  is convergent as a sequence of $n$. Hence the last term vanishes as $\gamma $ tends to zero for each $\varepsilon $ so that the second condition of Ardous' criteron (Proposition \ref{Aldous}) is satisfied. On the other hand, the first condition can be easily verified since an estimate $ \mathbb{E}_{\nu_\rho } [ \mathcal{M}^n_t (\varphi )^ 2] \le C t \| \varphi \|^2_{ L^2 (\mathbb{R} ) } $ ensures that the sequence $\{ \mathcal{M}^n_t (\varphi ) \}_{n \in \mathbb{N} } $ is uniformly bounded in $L^2 (\mathbb{P}_{\nu_\rho } ) $. Hence tightness of the martingale term is proved. 

\subsection{Symmetric part}
Next we show tightness of $\mathcal{S}^n $. Recalling the definition \eqref{eq:symm}, we have an expression 
\[
\mathcal{S}^n_t (\varphi ) 
= \frac{1}{ 2\sqrt{n} } \int_0^t \sum_{ j \in \mathbb{Z} } ( g_n ( \eta^n_j (s ) ) - \Phi_n (\rho) ) \Delta^n \varphi^n_j (s) ds 
\]
where we used the fact that the summation of $\Delta^n \varphi^n_j $ over $\mathbb{Z} $ equals to zero. Noting $E_{ \nu^n_\rho } [ g_n (\eta_j ) ] = \Phi_n (\rho) $, by Schwarz's inequality and stationarity, we have that 
\[
\begin{aligned}
\mathbb{E}_{ n } \big[ \big( \mathcal{S}^n_t (\varphi ) - \mathcal{S}^n_s (\varphi ) \big)^2 \big] 
& \le \frac{ t - s }{4 n } \mathbb{E}_n \bigg[ \int_s^t \bigg( \sum_{ j \in \mathbb{Z} } ( g_n (\eta^n_j ( r ) - \Phi_n (\rho ) ) \Delta^n \varphi^n_j (r ) \bigg)^2 dr \bigg] \\
& \le \frac{(t -s)^2}{ 4 } E_{\nu^n_\rho } \big[ ( g_n ( \eta_j ) - \Phi_n (\rho) )^2 \big] 
\sup_{ s \le r \le t  } \bigg( \frac{ 1 }{ n } \sum_{ j \in \mathbb{Z} } (\Delta^n \varphi^n_j  (r) )^2 \bigg) \\
& \le C  (t - s )^2 \| \partial^2_{ x } \varphi \|^2_{ L^2 (\mathbb{R} ) }     
\end{aligned}
\]
for every $s, t \in [0 , T] $ such that $s \le t $. Therefore by the Kolmogorov-Centsov criterion we conclude that the sequence $\{ \mathcal{S}^n_t (\varphi )  : t \in [0, T ] \}_{ n \in \mathbb{N} } $ is tight with respect to the uniform topology of $C ([0, T ] , \mathbb{R} ) $ and any limit point has $\alpha $-H\"{o}lder continuous trajectories with $\alpha < 1 /2 $.

\subsection{Anti-symmetric part}
Finally we consider the asymmetric part. By the second-order Boltzmann-Gibbs principle (Theorem \ref{BG2}) and stationarity, we have 
\[
\mathbb{E}_n \bigg[ \bigg| \tilde{\mathcal{B}}^n_t (\varphi) - \tilde{\mathcal{B}}^n_s (\varphi) - \int_s^t \sum_{ j \in \mathbb{Z} } \tau_j \mathcal{Q}^n_\rho (\ell ; r ) \nabla^n \varphi^n_j (r ) dr \bigg|^2 \bigg] 
\le C \| \partial_x \varphi \|^2_{ L^2 (\mathbb{R} ) }  \bigg( \frac{ (t- s ) \ell  }{ n } + \frac{ (t- s )^2 n }{ \ell^2 } \bigg) .
\]
On the other hand, by rearranging the sum as $\sum_{j} \tau_j \mathcal{Q} = \sum_{ i = 0, \ldots, \ell -1 } \sum_{ j  } \tau_{ j \ell + i } \mathcal{Q} $ recalling $\tau_j \mathcal{Q}$ and $\tau_k \mathcal{Q}$ are independent if $| j - k | \ge \ell $, we can show by an $L^2$-computation that 
\[
\mathbb{E}_n \bigg[ \bigg| \int_s^t  \sum_{ j \in \mathbb{Z} } \tau_j \mathcal{Q}^n_\rho (\ell ,r ) \nabla^n \varphi^n_j (r ) ds \bigg|^2 \bigg] \le C \| \partial_x \varphi \|^2_{ L^2 (\mathbb{R} ) } \frac{ (t -s )^2 n }{ \ell  } .
\]
Here we used $\mathbb{E}_n [\mathcal{Q}^n_\rho (\ell ; t )^2] = O (\ell^{-2 })$. Now we show tightness. First we consider the case $ t-s \ge 1/n^2 $. Then taking the scaling parameter $\ell $ proportional to $(t -s )^{ 1 / 2 } n $, we obtain  
\[
\mathbb{E}_n \big[ \big| \tilde{ \mathcal{B} }^n_t (\varphi) -  \tilde{ \mathcal{B} }^n_s (\varphi) \big|^2 \big] \le C \| \partial_x \varphi \|^2_{ L^2 (\mathbb{R} ) } (t -s )^{ 3 /2 } . 
\]
Next we shift to a short time regime $t- s \le 1/ n^2 $. Recalling the definition of $\tilde{\mathcal{J}}^n_t $, we have by a direct estimate 
\[
\mathbb{E}_n \big[ \big| \tilde{ \mathcal{B} }^n_t (\varphi) -  \tilde{ \mathcal{B} }^n_s (\varphi) \big|^2 \big] 
\le C \| \partial_x \varphi \|^2_{ L^2 (\mathbb{R} ) } (t -s )^2 n 
\le C \| \partial_x \varphi \|^2_{ L^2 (\mathbb{R} ) } (t -s )^{3 / 2}   . 
\]
Therefore combining the above estimates we conclude that the sequence $\{ \tilde{\mathcal{B}}^n_t (\varphi ) : t \in [0, T] \} $ is tight according to the Kolmogorov-Centsov criterion.

%%%%%%%%%%%%%%%%%%%%%%%%%%%%%%%%%%%%%%%%%
%   Identification of the limit point   %
%%%%%%%%%%%%%%%%%%%%%%%%%%%%%%%%%%%%%%%%%
\section{Identification of the limit point}
\label{sec:limitpt}
Again recall the martingale decomposition \eqref{decomposition}. We proved in Section \ref{sec:tightness} that the sequences $\{ \mathcal{X}^n_t : t \in [0, T] \}_{ n \in \mathbb{N} } $, $\{ \mathcal{M}^n_t : t \in [0, T] \}_{ n \in \mathbb{N} } $, $\{ \mathcal{S}^n_t : t \in [0, T] \}_{ n \in \mathbb{N} } $ and $\{ \tilde{ \mathcal{B} }^n_t : t \in [0, T] \}_{ n \in \mathbb{N} } $ are tight in $D ([0, T ] , \mathcal{S}^\prime (\mathbb{R} ) ) $ so that there exist processes $\mathcal{X}$, $\mathcal{M}$, $\mathcal{S}$ and $\tilde{ \mathcal{B} }$ such that 
\[
\begin{aligned}
 \lim_{ n \to \infty } \mathcal{X}^n = \mathcal{X}, \quad 
\lim_{ n \to \infty } \mathcal{M}^n = \mathcal{M}, \quad 
\lim_{ n \to \infty } \mathcal{S}^n = \mathcal{S}, \quad 
\lim_{ n \to \infty } \tilde{ \mathcal{B} }^n = \tilde{ \mathcal{B} } 
\end{aligned}
\]
in distribution along some subsequence that is still denoted by $n $. In the sequel, we characterize the limiting processes.

\subsection{Martingale part}
We decompose the quadratic variation of the martingale part, which is given by \eqref{qv}, as 
\[
\begin{aligned}
\langle \mathcal{M}^n (\varphi ) \rangle_t 
& = \frac{1}{n} \int_0^t \sum_{j \in \mathbb{Z} } g^\prime (0) \rho (\nabla^n \varphi^n_j (s) )^2 ds \\
& \quad + \frac{1}{n} \int_0^t \sum_{j \in \mathbb{Z} } [ \Phi_n (\rho) - g^\prime (0) \rho ] (\nabla^n \varphi^n_j (s) )^2 ds \\
& \quad + \frac{1}{n} \int_0^t \sum_{j \in \mathbb{Z} } [ g_n ( \eta^n_j (s) ) - \Phi_n ( \rho ) ] (\nabla^n \varphi^n_j (s) )^2 ds . 
\end{aligned}
\]
Then it is easy to show that $\langle \mathcal{M}^n (\varphi ) \rangle_t $ converges to $g^\prime (0) \rho \| \partial_x \varphi \|^2_{L^2 (\mathbb{R} ) } t $ in $L^2 (\mathbb{P}_n )$ as $n$ tends to infinity. Then by a similar way as \cite{gonccalves2015stochastic} we can show that the limiting process $\{ \mathcal{M}_t (\varphi) : t \in [0,T] \} $ is a martingale with quadratic variation $g^\prime (0) \rho \| \partial_x \varphi \|^2_{L^2 (\mathbb{R})} t$. 

%\begin{proposition}
%Let $\{ M^n_t  : t \in [0 , T ] \}_{ n \in \mathbb{N} } $ be a sequence of martingales converging in distribution to some process $\{ M_t : t \in [0, T ] \} $ as $n$ tends to infinity. If the sequence of random variables $\{ M^n_T \}_{ n \in \mathbb{N} } $ is uniformly integrable, then the process $\{ M_t : t \in [0, T ] \} $ is a martingale.   
%\end{proposition}

\subsection{Symmetric part}
Recalling that $g_n $ approaches to a linear function as $n$ tends to infinity, we decompose $\mathcal{S}^n$ as 
\[
\begin{aligned}
\mathcal{S}^n_t (\varphi ) 
= \frac{g^\prime (0) }{2\sqrt{n}} \int_0^t \sum_{j \in \mathbb{Z}} \eta^n_j (s) \Delta^n \varphi^n_j (s)  ds  
+ \frac{1}{2\sqrt{n}} \int^t_0 \sum_{j \in \mathbb{Z} } [g_n (\eta^n_j (s) ) - g^\prime(0) \eta^n_j (s) ] \Delta^n \varphi^n_j (s) ds .
\end{aligned}
\]
The second term converges to zero in $ L^2 (\mathbb{P}_n ) $ as $n$ tends to infinity. Thus the tightness of $\mathcal{X}^n_\cdot $ immediately shows 
\[
\mathcal{S}_t = \frac{ g^\prime (0) }{2} \int_0^t \mathcal{X}_s (\partial_x^2 \varphi ) ds .
\]

\subsection{Anti-symmetric part} 
We are in a position to identify the limit of anti-symmetric part $\mathcal{B}$. Here we define a modified version of the fluctuation field by 
\[
\tilde{ \mathcal{X} }^n_t (\varphi ) = \frac{1}{ \sqrt{n} } \sum_{ j \in \mathbb{Z} } \overline{W}_j (t) \varphi^n_j (t)  
\]
for each $\varphi \in \mathcal{S} (\mathbb{R}) $. Recall the definition of $\mathcal{Q}^n_\rho (\ell ; t ) $ given in \eqref{Q} and the function $\iota_{\varepsilon } (x ; \cdot) : \mathbb{R} \to \mathbb{R} $ defined by $\iota_\varepsilon (x; y ) = \varepsilon^{ - 1 } \mathbf{1}_{ [ x, x + \varepsilon ) } (y) $ for each $x \in \mathbb{R}$. Hereafter we use an abuse of notation to denote the integer part of $\varepsilon n $ by the same notation. Then one can notice $\tilde{ \mathcal{X} }^n_t (\iota_{\varepsilon } ( \frac{ j - f_n t }{ n }  ; \cdot ) ) = \sqrt{n} \overrightarrow{W}^\ell_j (t) $ for each $j \in \mathbb{Z} $ so that 
\[
\sum_{ j \in \mathbb{Z} } \tau_j \mathcal{Q}^n_\rho (\varepsilon n ; t ) \nabla^n \varphi^n_j (t ) 
=  \frac{ g^{ \prime \prime } (0) }{ 2 n } \sum_{ j \in \mathbb{Z} } \big( \tilde{ \mathcal{X} }^n_t ( { \textstyle \iota_{ \varepsilon } ( \frac{ j- f_n t }{ n } ; \cdot ) } ) \big)^2 \nabla^n \varphi^n_j (t ) . 
\]
Then, letting $n $ tends to infinity, we obtain the limit 
\[
\frac{g^{\prime \prime }(0)}{2} \mathcal{A}^\varepsilon_{ s,t } (\varphi ) = \lim_{ n \to \infty } \int_s^t \sum_{ j \in \mathbb{Z} } \tau_j \mathcal{Q}^n_\rho (\varepsilon n ; r ) \nabla^n \varphi^n_j (r ) dr  
\]
in $L^2 (\mathbb{P}_n ) $. Note that such a limiting procedure does not hold immediately since the function $\iota_\varepsilon (x, \cdot) $ is not in $\mathcal{S} (\mathbb{R})$. However, wee can approximate it by functions in $\mathcal{S} (\mathbb{R})$ and we can justify the convergence (see \cite{gonccalves2014nonlinear} for detail). 

By Theorem \ref{BG2}, we have
\[
\mathbb{E}_n \bigg[ \bigg| \tilde{ \mathcal{B} }_t (\varphi ) - \tilde{\mathcal{B}}_s (\varphi ) - \int_s^t \sum_{ j \in \mathbb{Z} } \tau_j \mathcal{Q}^n_\rho (\ell, r ) \nabla^n \varphi^n_j (r) dr \bigg|^2 \bigg] 
\le C \| \partial_x \varphi \|^2_{ L^2 (\mathbb{R} ) }\bigg( \frac{ (t -s ) \ell }{ n } + \frac{ (t-s )^2 n }{ \ell^2 }  \bigg) .
\] 
Now taking $\ell = \varepsilon n $ and then letting $n \to \infty $, we obtain 
\begin{equation}
\label{B-A}
\mathbb{E}_n \bigg[ \bigg|  \tilde{ \mathcal{B} }_t (\varphi ) - \tilde{\mathcal{B}}_s (\varphi ) - \frac{g^{\prime \prime}(0)}{2}\mathcal{A}^\varepsilon_{ s , t } (\varphi ) \bigg|^2 \bigg] \le C  \varepsilon (t - s ) \| \partial_x \varphi \|^2_{ L^2 (\mathbb{R} ) } ,
\end{equation}
from which we get the condition \textbf{(EC2)} with the help of the triangle inequality. Hence Proposition \ref{nonlinear} bring us the existence of the limit $\mathcal{A}_t (\varphi ) = \lim_{ \varepsilon \to 0 } \mathcal{A}^\varepsilon_{0, t } (\varphi )  $ for each $\varphi \in \mathcal{S} (\mathbb{R} ) $. In addition, the estimate \eqref{B-A} assures $\tilde{ \mathcal{B} } = (g^{\prime \prime}(0)/2) \mathcal{A}$.

Moreover, we show that also the condition \textbf{(EC1)} holds true. For that purpose, it suffices to check that 
\[
\mathbb{E}_n \bigg[ \bigg| \int_0^t \tilde{ \mathcal{ X } }^n_s (\partial_x^2 \varphi ) ds \bigg|^2 \bigg] \le C t \| \partial_x^2 \varphi \|^2_{ L^2 (\mathbb{R}) } .
\]
Furthermore, according to summation by parts and smoothness of each test function $\varphi $, it is enough to verify that  
\[
\mathbb{E}_n \bigg[ \bigg| \int_0^t \sqrt{n } \sum_{ j \in \mathbb{Z} } (\overline{W}_{j - 1 } (s ) - \overline{W}_j (s) ) \nabla^n \varphi^n_j (s) ds \bigg|^2 \bigg] \le C t \| \partial_x \varphi \|^2_{ L^2 (\mathbb{R}) } .
\] 
This follows by and $H^{ - 1 , n }  $-estimate as we conducted in Section \ref{sec:BG2}. Indeed, by an integration-by-parts formula given in Lemma \ref{IBP}, we have 
\[
\begin{aligned}
2 \bigg\langle \sqrt{n } \sum_{ j \in \mathbb{Z} } (\overline{W}_{j - 1 } - \overline{W}_j ) \nabla^n \varphi^n_j  , f \bigg\rangle_{ L^2 (\nu^n_\rho ) } 
= -  2 \sqrt{n} E_{ \nu^n_\rho } \bigg[ \sum_{ j\in \mathbb{Z} } W_{ j - 1 } \nabla^n \varphi^n_j \nabla_{ j-1 , j  } f  \bigg] .
\end{aligned}
\]
By Young's inequality, recalling the definition of $W_j$, the last display is absolutely bounded above by 
\[
\frac{n^2}{2} \sum_{ j \in \mathbb{Z} } E_{ \nu^n_\rho }  [ g_n (\eta_j) (\nabla_{ j , j + 1 } f )^2 ] 
+ \frac{2}{g^\prime (0) n}\sum_{ j \in \mathbb{Z} } E_{ \nu^n_\rho }  [ W_{ j - 1 }  (\nabla^n \varphi^n_j  )^2 ]  .
\]
The first term is bounded by  $\|f \|_{ 1 , n } $ so that we conclude that the condition \textbf{(EC1)} holds true by the Kipnis-Varadhan inequality. 

Finally we note that all the above estimates hold also for the reversed process $\{ \mathcal{X}^n_{T-t} : t \in [0, T] \}$ and thus condition (3) of Definition \ref{def:energysol} is satisfied. By this line it is prove that the limiting process $\mathcal{X} $ is the energy solution of the stochastic Burgers equation \eqref{SBEthm}.

%%%%%%%%%%%%%%%%%%%%%%%%%%%%%%%%%%%%%%%%%%
%  Taylor expansion of $ W_j - \eta_j $  %
%%%%%%%%%%%%%%%%%%%%%%%%%%%%%%%%%%%%%%%%%%
\section{Proof of Lemma \ref{antisymm}}
\label{sec:expansion}
Finally in this section we give the proof of Lemma \ref{antisymm} by showing how to take the framing in a suitable way. Recall the definition of $W_j $ given in \eqref{eq:wdef}. Then by a Taylor expansion we have that
\[
W_j = \eta_j + \frac{ g^{\prime \prime} (0 ) }{ 2 g^\prime (0 ) } \frac{ \eta_j^2 }{ \sqrt{n } } + \frac{ g^{ (3) } (0) }{ 6 g^\prime (0) } \frac{ \eta_j^3 }{ n } + O (n^{ - 3 /2 } ) . 
\]
By Assumption \ref{ass:regularity}, a direct $L^2$ computation enables us to estimate the reminder term $O(n^{-3/2})$ as follows. 
%can be written as $n^{ - 3 /2 } p_n (\eta) $ with some local function $p_n $ whose second moment stays finite uniformly in $n$, which can be neglected according to the following direct computation: 
\[
\mathbb{E}_n \bigg[ \sup_{ 0 \le t \le T } \bigg|  \sqrt{n} \int_0^t \sum_{ j \in \mathbb{Z} } ( O(n^{ - 3 /2 }) - E_{ \nu^n_\rho } [ O (n^{-3/2} ) ] ) \nabla^n \varphi^n_j (s) ds \bigg|^2 \bigg] \le C \frac{T^2}{n} \| \partial_x \varphi \|^2_{ L^2 (\mathbb{R} ) } ,
\] 
which vanishes as $n $ tend to infinity. In particular, the reminder term may not be concerned. In the sequel, we write the difference $W_j - \eta_j $ in terms of the variables $W_j$'s. For that purpose, note that
\[
\frac{a}{\sqrt{n}} W_j \eta_j + \frac{b}{n} W_j \eta_j^2 
= a \frac{\eta_j^2}{\sqrt{n}} + \bigg( a \frac{ g^{ \prime \prime } (0 ) }{ 2 g^\prime (0 ) } +  b  \bigg) \frac{\eta_j^3}{n} + O (n^{-3/2}) .
\] 
We take the constants $a$ and $b$ in order that leading terms in the right-hand side of the above display coincide with those of $W_j - \eta_j $. Then we have 
\[
W_j - \eta_j = \frac{ g^{\prime \prime } (0) }{ 2 g^\prime (0) } \frac{ W_j \eta_j }{ \sqrt{n } } 
+ \bigg( \frac{ g^{ (3) } (0 ) }{ 6 g^\prime (0) } - \frac{ g^{ \prime \prime } (0 )^2 }{ 4  g^\prime (0)^2 } \bigg) \frac{ W_j \eta_j^2 }{ n } + O (n^{ - 3/2 }) .  
\]
Next we notice the following identities.
\[
\begin{aligned}
%& \tilde{L}_n \eta_j  = W_{ j - 1 } - W_j  , \\
& \tilde{L}_n ( \eta_j^2 - \eta_j ) = 2 ( W_{ j -1 } - W_j ) \eta_j + 2  W_j  , \\
& \tilde{L}_n \big( \eta_j^3 + \frac{3}{2} \eta_j^2 + \frac{ 5 }{ 2 } \eta_j  \big) = 3  (W_{ j -1 } - W_j ) \eta_j^2 + 6 W_{ j - 1 } \eta_j + 3 W_j  
\end{aligned} 
\]
where $\tilde{L}_n = ( n^2 g^\prime (0 ) )^{ -1} L_n $. Here we have the following result. 

\begin{lemma}
\label{lem:range}
Let $G$ be a local function such that $E_{\nu^n_\rho } [g(\eta_j ) G(\eta)] $ is finite for any $j \in \mathbb{Z} $ and we write $G_j = \tau_j G$. Then we have 
\[
\mathbb{E}_n \bigg[ \sup_{0 \le t \le T } \bigg| \int_0^t \sum_{ j \in \mathbb{Z} } n^{-2} L_n G_j (\eta^n (s) ) \nabla^n \varphi^n_j (s) ds \bigg|^2 \bigg] \le \frac{C}{n^2} \int_0^T \sum_{j \in \mathbb{Z}} ( \nabla^n \varphi^n_j (s) )^2 ds .
\]
\end{lemma}
\begin{proof}
By definition of the adjoint operator $L^*_n $, for any local $L^2 (\nu^n_\rho ) $ funciton $f$, we have that 
\[
\begin{aligned}
 \bigg\langle \sum_{j \in \mathbb{Z}} n^{-2} L_n G_j \nabla^n \varphi^n_j , f \bigg\rangle_{L^2(\nu^n_\rho)} 
& = \sum_{j \in \mathbb{Z}} \langle G_j \nabla^n \varphi^n_j, n^{-2} L_n^* f \rangle_{L^2(\nu^n_\rho)} \\
& = \sum_{\substack{ j, k \in \mathbb{Z}, \\ k \in \mathrm{supp}G_j \cup (\mathrm{supp}G_j +1 ) } } E_{ \nu^n_\rho } [ G_j (\eta ) \nabla^n \varphi^n_j g(\eta_k ) \nabla_{ k,k- 1 } f (\eta)] .
\end{aligned}
\]
However, by Young's inequality, twice the last display is bounded by 
\[
\begin{aligned}
& 2n^{-2} (| \mathrm{supp}G_j | + 1 ) \sum_{\substack{ j, k \in \mathbb{Z}, \\ k \in \mathrm{supp}G_j \cup (\mathrm{supp}G_j+1 ) } } E_{ \nu^n_\rho } [ g(\eta_k ) G_j (\eta )^2 ](\nabla^n \varphi^n_j)^2 \\ 
& \quad + \frac{n^2}{2} \sum_{ k \in \mathbb{Z} } E_{ \nu^n_\rho } [ g(\eta_k ) ( \nabla_{ k,k- 1 } f (\eta) )^2 ] .
\end{aligned}
\]
Since the second term is bounded by $\| f \|^2_{1, n } $, we complete the proof by using the Kipnis-Varadhan inequality. 
\end{proof}
Lemma \ref{lem:range} assures that quantity contained in the range of $\tilde{L}_n $ is small and thus we can replace $W_{ j } \eta_j  $ and $W_j \eta_j^2 $ in terms of $W_{ j -1 } \eta_j $ and $W_{j -1 } \eta_j^2 $ without any trouble. Moreover, again by a Taylor expansion, we have 
\[
\frac{ W_{ j - 1 } W_j }{ \sqrt{n} } = \frac{ W_{ j -1 } \eta_j }{ \sqrt{n} }
+ \frac{ g^{\prime \prime } (0) }{ 2 g^\prime ( 0 ) } \frac{ W_{ j -1 } \eta_j^2  }{ n } 
%+ \frac{ g^{ (3) } (0) }{ 6 g^\prime ( 0 ) } \frac{ W_{ j -1 } \eta_j^3 }{ n }  
+ O ( n^{ - 3/2 } ) .
\]
Hence by this line we obtain  
\begin{equation}
\label{eq:expansion1}
\begin{aligned}
W_j - \eta_j 
=& \bigg( \frac{ g^{\prime \prime} (0) }{ 2 g^\prime (0) \sqrt{ n } } 
+ \frac{ g^{ (3 ) } (0) }{ 3 g^\prime (0) n } 
- \frac{ g^{\prime \prime} (0)^2 }{ 2 g^\prime (0)^2 n } \bigg) W_{ j - 1 } W_j  
+ \bigg( \frac{ g^{\prime \prime} (0) }{ 2 g^\prime (0) \sqrt{ n } } 
+ \frac{ g^{ (3 ) } (0) }{ 6 g^\prime (0) n } 
- \frac{ g^{\prime \prime} (0)^2 }{ 4 g^\prime (0)^2 n } \bigg) W_j \\
& + \bigg( \frac{ g^{ (3 ) } (0) }{ 6 g^\prime (0) n } 
- \frac{ 3 g^{\prime \prime} (0)^2 }{ 4 g^\prime (0)^2 n } \bigg) W_{j - 1 } \eta_j^2  
+ R_j 
\end{aligned}
\end{equation}
where $R_j $ is a negligible error term. Finally we just need to estimate the order-three term. Since the coefficient of $W_{j-1} \eta^2_j$ in the above identity has order $O(n^{-1})$, we can replace $ \eta_j $ by $ W_j $ by using the same identity. To calculate further, we use identities 
\[
\begin{aligned}
\tilde{L}_n (\eta_{ j - 1 } \eta_j \eta_{ j +  1} ) 
& =  W_{ j -2} W_j W_{ j + 1 } + W_{ j - 1 }^2 W_{ j + 1} +W_{ j - 1} W_j^2 -3 W_{ j -1 }W_j W_{ j + 1 } \\
& \quad  - W_{ j - 1 } W_{ j } - W_{ j - 1 } W_{ j + 1} + E^{(1)}_{ j } , \\
\tilde{L}_n (\eta_{ j - 1 }^2 \eta_{ j +  1} ) 
& =  2 W_{ j -2} W_{j - 1 } W_{ j + 1 }  + W_{ j - 1}^2 W_j - 3 W_{ j - 1 }^2 W_{ j + 1}   \\
& \quad + W_{ j - 2 } W_{ j + 1 } + W_{ j - 1 } W_{j + 1 } + E^{(2)}_{ j } , \\
\tilde{L}_n (\eta_{ j - 1 }^2  \eta_j  ) 
& = 2 W_{ j -2} W_{j - 1 } W_{ j } - 3 W_{ j - 1 }^2 W_{ j } + W_{ j - 1}^3  \\
& \quad + W_{j-2} W_{j} + W_{ j - 1 } W_{ j} - 2 W_{j-1}^2  + W_{j-1} + E^{(3)}_{ j } ,\\
\tilde{L}_n \eta_{ j  }^3 
& = 3 W_{ j -1 } W_{ j }^2 - 3 W_j^3 \\
& \quad + 3 W_{ j -1 } W_j  
 + 3 W_{ j }^2 + W_{j-1}- W_j + E^{( 4 ) }_{ j } 
\end{aligned}
\]
with some remainder terms $E^{(k)}_j $, which do not affect the limit. Here note that we can shift index of each term since for any local function $G$ we have that 
\[
\begin{aligned}
& \mathbb{E}_n \bigg[ \sup_{0\le t \le T } \bigg|
\frac{1}{\sqrt{n}} \int_0^t \sum_{j \in \mathbb{Z}} (\tau_j G -\tau_{j-1}G)(s)\nabla^n \varphi^n_j(s) ds \bigg|^2 \bigg] \\
& \quad = \mathbb{E}_n \bigg[ \sup_{0\le t \le T } \bigg|
\frac{1}{\sqrt{n}} \int_0^t \sum_{j \in \mathbb{Z}} \tau_j G(s) (\nabla^n \varphi^n_j - \nabla^n \varphi^n_{j+1})(s) ds \bigg|^2 \bigg]
\le C \frac{T^2}{n} E_{\nu^n_\rho}[G^2] \| \partial_x^2 \varphi \|^2_{L^2(\mathbb{R})} 
\end{aligned}
\]
where in the last inequality we used Schwarz's inequality. Moreover for order-two terms, we have that
\[
\begin{aligned}
\tilde{L}_n (\eta_{ j - 1 } \eta_j ) 
& =  W_{ j -2 } W_{j} +W_{j-1}^2 - 2W_{ j - 1 } W_j - W_{j-1}  + E^{(5)}_{ j } , \\
\tilde{L}_n (\eta_{ j - 1 } \eta_{ j +  1} ) 
& = W_{ j -2} W_{ j + 1 } - 2 W_{ j - 1 } W_{ j + 1} + W_{ j - 1} W_j  + E^{(6)}_{ j } , \\
\tilde{L}_n \eta_j^2  
& = 2 W_{ j - 1} W_j - 2 W_j^2 + W_{j-1} + W_j + E^{(7)}_{ j } .
\end{aligned}
\]
From these identities and shifting indices, we may replace $W_{j-2}W_{j}$, $W_{j-2}W_{j+1}$ and $W_{j-1}W_{j+1}$ by $W_{j-1} W_j$, and $W_{j-1}^2$ and $W_j^2$ by $W_{j-1}W_j + W_j$, respectively. As summary, we can represent the order-three term $W_{ j - 1 } W_j^2 $ as 
\begin{equation}
\label{eq:expansion2}
\begin{aligned}
10 W_{ j - 1 } W_j^2 = 
& - 2 W_{ j - 2 } W_{ j - 1 } W_j - 6 W_{ j - 2 } W_{ j - 1 } W_{ j + 1 } - 9 W_{ j - 2 } W_{ j } W_{ j + 1 } + 27 W_{ j - 1 } W_j W_{ j + 1 } \\
&  + 10 W_{ j - 1 } W_{ j } + E_j^{(8)} . 
% Calculate - 27 \eta_{j-1} \eta_j \eta_{j+1} -9 \eta_{j-1}^2 \eta_{j+1} - 3\eta_{j-1}^2 \eta_j - \eta_j^3 
\end{aligned}
\end{equation}
Hence from now on, we are concerned with order-three term with distinct indices. For that purpose, we first show the following estimate. 

\begin{lemma}
\label{lem:order3}
We have that
\[
\begin{aligned}
& \mathbb{E}_n \bigg[ \sup_{0 \le t \le T} \bigg| \frac{1}{\sqrt{n}} \int_0^t \sum_{j \in \mathbb{Z} }
\big\{ \overline{W}_{j-1}(s) \overline{W}_j(s) \overline{W}_{j+1}(s) 
- (\overrightarrow{W}^\ell_{j+1}(s) )^3
\big\} \nabla^n \varphi^n_j (s) ds \bigg|^2 \bigg] \\
& \quad \le C \bigg( \frac{T \ell}{n^2} + \frac{T^2}{\ell^2} \bigg) \| \partial_x \varphi \|^2_{L^2(\mathbb{R})} . 
\end{aligned}
\]
\end{lemma}
\begin{proof}
We use the decomposition
\[
\begin{aligned}
& \overline{W}_{j-1} \overline{W}_j \overline{W}_{j+1} 
- (\overrightarrow{W}^\ell_{j+1} )^3 \\
& \quad = 
\overline{W}_{j-1} \bigg( \overline{W}_j \overline{W}_{j+1} - (\overrightarrow{W}^\ell_{j+1})^2 + \frac{\sigma_n^2(\rho) }{\ell} \bigg) \\
& \qquad + (\overrightarrow{W}_{j+1}^\ell)^2 (\overline{W}_{j-1} - \overrightarrow{W}^\ell_j) 
+ (\overrightarrow{W}_{j+1}^\ell)^2 (\overrightarrow{W}_j - \overrightarrow{W}^\ell_{j+1}) .
\end{aligned}
\]
For the first term in the right hand side, noting $\overline{W}_{j-1}$ is independent of the quantity inside the parenthesis, we can conduct the same procedure as Theorem \ref{BG2}, which gives the bound $C (T \ell n^{-2} + T^2 \ell^{-2}) \| \partial_x \varphi \|^2_{L^2(\mathbb{R})}$. In addition, the third term can be treated by a direct $L^2$-estimate and we obtain the bound $CT^2\ell^{-2} \| \partial_x \varphi \|^2_{L^2(\mathbb{R})}$. For the second term, note that  
\[
(\overrightarrow{W}_{j+1}^\ell)^2 (\overline{W}_{j-1} - \overrightarrow{W}^\ell_j)  
= (\overrightarrow{W}_{j+1}^\ell)^2 
\sum_{i=0}^{\ell-1} \psi_i (W_{j+i-1} - W_{j+i}) 
\]
where $\psi_i = (\ell-i)/\ell$. Then the integration-by-parts formula (Lemma \ref{IBP}) yields 
\[
\begin{aligned}
& E_{\nu^n_\rho} [f(\overrightarrow{W}_{j+1}^\ell)^2 (W_{j+i-1} - W_{j+i})] \\
& \quad = - E_{\nu^n_\rho} [\sigma_{j+i-1, j+i}(\overrightarrow{W}_{j+1}^\ell)^2
(\nabla_{j+i-1, j+i} f) W_{j+i-1}] 
- E_{\nu^n_\rho} [f (\nabla_{j+i-1, j+i}(\overrightarrow{W}_{j+1}^\ell)^2) W_{j+i-1}] 
\end{aligned}
\]
for any local function $f$. However, by Young's inequality, we have that 
\begin{equation}
\label{eq:estimateorder3}
\begin{aligned}
& 2 \bigg\langle \sum_{j \in \mathbb{Z} } 
\sum_{i=0}^{\ell-1} F^\ell_{i,j } \psi_i \nabla^n \varphi^n_j 
, f \bigg\rangle_{L^2(\nu^n_\rho)} \\
& \quad \le
\frac{n^2 g^\prime(0)}{2\ell}\sum_{j \in \mathbb{Z}} \sum_{i=0}^{\ell-1} E_{\nu^n_\rho} [W_{j+i-1} (\nabla_{j+i-1, j+i}f)^2] \\
& \qquad 
+ \frac{2 \ell}{n^2 g^\prime(0)} \sum_{j \in \mathbb{Z}} \sum_{i=0}^{\ell-1} 
E_{\nu^n_\rho}[W_{j+i-1} \big(\sigma_{j+i-1, j+i}(\overrightarrow{W}_{j+1}^\ell)^2\big)^2] (\nabla^n \varphi^n_j )^2
\end{aligned}
\end{equation}
where we set 
\[
F^\ell_{i,j}
= (\overrightarrow{W}_{j+1}^\ell)^2 (W_{j+i-1} - W_{j+i})
+ (\nabla_{j+i-1, j+i}(\overrightarrow{W}_{j+1}^\ell)^2) W_{j+i-1} .
\]
Since the first term in the right hand side of \eqref{eq:estimateorder3} is bounded by $\| f \|^2_{1,n }$ and
\[
E_{\nu^n_\rho}[W_{j+i-1} \big(\sigma_{j+i-1, j+i}(\overrightarrow{W}_{j+1}^\ell)^2\big)^2] = O(\ell^{-2}) ,
\]
the Kipnis-Varadhan inequality yields 
\[
\mathbb{E}_n \bigg[ \sup_{0 \le t \le T} \bigg| \frac{1}{\sqrt{n}} \int_0^t \sum_{j \in \mathbb{Z}} \sum_{i=0}^{\ell-1} F^\ell_{i,j}(s) \psi_i \nabla^n \varphi^n_j (s) ds \bigg|^2 \bigg]
\le C \frac{T }{n^2} \| \partial_x \varphi \|^2_{L^2(\mathbb{R})}. 
\]
Therefore, it remains to estimate 
\[
\mathbb{E}_n \bigg[ \sup_{0 \le t \le T} \bigg| \frac{1}{\sqrt{n}} \int_0^t \sum_{j \in \mathbb{Z}} \sum_{i=0}^{\ell-1} 
(\nabla_{j+i-1, j+i}(\overrightarrow{W}_{j+1}^\ell(s))^2) W_{j+i-1}(s) \psi_i \nabla^n \varphi^n_j (s) ds \bigg|^2 \bigg] .
\]
Here we note that 
\[
\begin{aligned}
 \nabla_{j+i-1, j+i} (\overrightarrow{W}^\ell_{j+1})^2 
= \ell^{-1} \overrightarrow{W}^\ell_{j+1} 
\nabla_{j+i-1, j+i} Z_{i,j}  
- \ell^{-2} (\nabla_{j+i-1, j+i} Z_{i,j} )^2 
\end{aligned}
\]
provided $\eta_{j+i-1} > 0$ where $Z_{i,j} = W_{j+i-1}\mathbf{1}_{i \ge 1} + W_{j+i} \mathbf{1}_{i \ge 0}$. 
The second term gives small factor by a direct $L^2$ computation so that we may consider only the first term. Moreover, a change of variable yields 
\[
\begin{aligned}
&  E_{\nu^n_\rho} [ f \overrightarrow{W}^\ell_{j+1} \nabla_{j+i-1, j+i} Z_{i,j} W_{j+i-1} ] \\
& \quad  
= E_{\nu^n_\rho} [ \nabla_{j+i-1, j+i} (f  \overrightarrow{W}^\ell_{j+1} ) Z_{i,j} W_{j+i} ] 
%\\ & \qquad  
- E_{\nu^n_\rho} [ f \overrightarrow{W}^\ell_{j+1} ( W_{j+i-1} - W_{j+i} ) Z_{i,j} ] \\
& \quad
= E_{\nu^n_\rho} [ (\nabla_{j+i-1, j+i} f) (\sigma_{j+i-1, j+i} \overrightarrow{W}^\ell_{j+1} ) Z_{i,j} W_{j+i} ]
%\\ & \qquad 
- E_{\nu^n_\rho} [ f (\nabla_{j+i-1, j+i}  \overrightarrow{W}^\ell_{j+1} ) Z_{i,j} W_{j+i} ] \\
& \qquad
- E_{\nu^n_\rho} [ f \overrightarrow{W}^\ell_{j+1}(W_{j+i-1}-W_{j+i}) Z_{i,j}] .
\end{aligned}
\]
The first two terms can be estimated by the Kipnis-Varadhan inequality as the above and thus we may only consider the last term in the right-hand side. However, note that  
\[
\begin{aligned}
& \frac{1}{\ell} \sum_{i=0}^{\ell-1} \overrightarrow{W}^\ell_{j+1}Z_{i,j}(W_{j+i-1}-W_{j+i}) \psi_i \\
& \quad = \frac{1}{\ell} \overrightarrow{W}^\ell_{j+1} (W^2_{j-1} - E_{\nu^n_\rho}[W^2_{j-1}] - \overrightarrow{(W^2)}^\ell_j ) 
- \frac{\ell-1}{\ell^2} W_j(W_j- W_{j+1}) 
- \frac{1}{\ell} W_{j-1} (W_{j-1} - W_j) .
\end{aligned}
\]
All terms in the right-hand side can be treated by a direct $L^2$ computation, which gives the bound $C T^2 \ell^{-2} \|\partial_x \varphi\|^2_{L^2(\mathbb{R})}$. Hence we obtained the desired estimate and complete the proof. 
\end{proof}

Now we see that the order-three terms with distinct indices do not affect the limit if they are centered. When $T \ge 1/n^2$, taking $\ell \sim \sqrt{T}n$, the bound in Lemma \ref{lem:order3} becomes $CT^{3/2} n^{-1} \|\partial_x \varphi \|^2_{L^2(\mathbb{R})}$. Moreover, since $E[(\overrightarrow{W}^\ell_j)^6]=O(\ell^{-3})$, we have that 
\[
\mathbb{E}_n \bigg[ \sup_{0 \le t \le T} \bigg| \frac{1}{\sqrt{n}} \int_0^t \sum_{j \in \mathbb{Z} }
(\overrightarrow{W}^\ell_{j+1}(s) )^3
\nabla^n \varphi^n_j (s) ds \bigg|^2 \bigg]
\le C \frac{T^2}{\ell^2} \| \partial_x \varphi \|^2_{L^2(\mathbb{R})}
\le C \frac{T^{3/2}}{n} \| \partial_x \varphi \|^2_{L^2(\mathbb{R})} .
\]
Combining both estimates, we have a bound 
\[
\mathbb{E}_n \bigg[ \sup_{0 \le t \le T} \bigg| \frac{1}{\sqrt{n}} \int_0^t \sum_{j \in \mathbb{Z} }
\overline{W}_{j-1}(s) \overline{W}_{j}(s) \overline{W}_{j+1}(s) 
\nabla^n \varphi^n_j (s) ds \bigg|^2 \bigg]
\le C \frac{T^{3/2}}{n} \| \partial_x \varphi \|^2_{L^2(\mathbb{R})} .
\]
On the other hand, for the case $T < 1/n^2$, a direct $L^2$ computation yields 
\[
\begin{aligned}
& \mathbb{E}_n \bigg[ \sup_{0 \le t \le T} \bigg| \frac{1}{\sqrt{n}} \int_0^t \sum_{j \in \mathbb{Z} }
\overline{W}_{j-1}(s) \overline{W}_{j}(s) \overline{W}_{j+1}(s) 
\nabla^n \varphi^n_j (s) ds \bigg|^2 \bigg] \\
& \quad 
\le C T^2 \| \partial_x \varphi \|^2_{L^2(\mathbb{R})} 
\le C \frac{T^{3/2}}{n} \| \partial_x \varphi \|^2_{L^2(\mathbb{R})} .
\end{aligned}
\]
We can control $\overline{W}_{j-2}\overline{W}_{j-1}\overline{W}_{j}$, $\overline{W}_{j-2}\overline{W}_{j-1}\overline{W}_{j+1}$ and $\overline{W}_{j-2}\overline{W}_{j}\overline{W}_{j+1}$ in the same way. 

Now we return to expand $W_j - \eta_j$ in terms of $W_j$'s. In the expansion \eqref{eq:expansion1}, we can make indices of order-three terms distinct by using \eqref{eq:expansion2}. Moreover, note that when replacing $W_j$ by its centered version $\overline{W}_j = W_j - \Phi_n(\rho)$, terms of the form $(1/n) \overline{W}_{j_1} \overline{W}_{j_2}$ with $j_1\neq j_2$ and constants do not affect the limit. In other words, in \eqref{eq:expansion1} and \eqref{eq:expansion2}, we can replace $W_{j-1}W_j W_{j+1}$ and $W_{j-1}W_j$ by $\overline{W}_{j-1}\overline{W}_{j}\overline{W}_{j+1} + 3\Phi_n(\rho)^2 W_j$ and $\overline{W}_{j-1}\overline{W}_{j} + 2\Phi_n(\rho) W_j$, respectively. As a consequence, $W_j - \eta_j $ can be expanded as  
\begin{equation}
\label{expansion}
\frac{ g^{ \prime \prime } (0) }{ 2 g^\prime (0) } \frac{ \overline{W}_{ j -1 } \overline{ W}_j }{ \sqrt{n} } 
+ \bigg(  \frac{b_1}{ \sqrt{n} } + \frac{b_0}{n} \bigg) \eta_j 
\end{equation}
plus some asymptotically small factor. Here the constants $b_0$ and $b_1$ is of the form \eqref{eq:framing}. This expansion immediately completes the proof of Lemma \ref{antisymm}.

\section*{Acknowledgments}
The author would like to thank Makiko Sasada and Hayate Suda for giving him fruitful comments and discussions.

\bibliographystyle{abbrv}
\bibliography{reference}

\end{document}